\documentclass[11pt, reqno]{amsart}
\usepackage[margin=3cm]{geometry}
\usepackage{graphicx} 
\usepackage{comment}
\usepackage{adjustbox}
\usepackage{bm}
\usepackage{tabularx}
\usepackage{arydshln,leftidx,mathtools}
\usepackage{tikz}
\usetikzlibrary{positioning}

\usepackage{amsmath,amssymb,amsxtra,tikz,stackengine,bm,mathrsfs,bbm,tikz-cd,comment,hyperref,combelow,amsopn,float}
\usepackage{enumitem}
\usetikzlibrary{calc}%
\usetikzlibrary{shapes}%
\usetikzlibrary{patterns}%
\usetikzlibrary{positioning}%
\usetikzlibrary{arrows.meta}
\usetikzlibrary{knots}
\usetikzlibrary{decorations.markings}
\usetikzlibrary{hobby}
\usepackage{xcolor}

\newtheorem{thm}{Theorem}[section]
\newtheorem{lem}[thm]{Lemma}

\newtheorem{cor}[thm]{Corollary}
\theoremstyle{definition}
\newtheorem{exam}[thm]{Example}
\newtheorem{defn}[thm]{Definition}

\newtheorem{problem}[thm]{Problem}
\newtheorem{remark}[thm]{Remark}

\usetikzlibrary{calc}%
\usetikzlibrary{shapes}%
\usetikzlibrary{patterns}%
\usetikzlibrary{positioning}%
\usepackage{subfig}

\newcommand{\Br}{\mathrm{Br}_{n}^{+}}

\newcommand{\X}{X_{u,{\bm{\beta}}}}
\newcommand{\C}{\mathbb{C}}
\newcommand{\CC}{\mathbb{C}}
\newcommand{\Z}{\mathbb{Z}}

\newcommand{\bhat}{\widehat{B^{u,\bm{\beta}}}}
\newcommand{\ahat}{{A^{u,\bm{\beta}}}}
\newcommand{\til}{\widetilde{B}^{u,\bm\beta}}
\newcommand{\Hmatrix}{H^{u,\bm{\beta}}}
\newcommand{\D}{D^{u,\bm{\beta}}}

\newcommand{\p}{\partial}
\newcommand{\newword}[1]{\emph{\textbf{#1}}}
\newtheorem{mainthm}{Theorem}

\definecolor{cinnamon}{rgb}{0.82, 0.41, 0.12}
\definecolor{babyblue}{rgb}{0.54, 0.81, 0.94}
\definecolor{lightskyblue}{rgb}{0.53, 0.81, 0.98}
\definecolor{silver}{rgb}{0.2, 0.2, 0.2}
\definecolor{almond}{rgb}{0.94, 0.87, 0.8}
\definecolor{azure}{rgb}{0.94, 1.0, 1.0}
\definecolor{piggypink}{rgb}{0.99, 0.93, 0.96}
\definecolor{frenchrose}{rgb}{0.96, 0.29, 0.54}

\title{Cluster automorphism group of braid varieties}
\author{Soyeon Kim}
\address{Department of Mathematics\\ University of California at Davis\\ One Shields Avenue, Davis CA 95616}
\email{syxkim@ucdavis.edu}
\date{}

\begin{document}

\begin{abstract}
The cluster automorphism group of a cluster variety was defined by Gekhtman--Shapiro--Vainshtein \cite{GSV}, and later studied by Lam--Speyer \cite{LS22}. Braid varieties are interesting affine algebraic varieties indexed by positive braid words. It was proved recently that braid varieties are cluster varieties.
In this paper, we propose an explicit description of the cluster automorphism group and its action on braid varieties, and provide several computed examples.
\end{abstract}

\maketitle

\section{Introduction}

A cluster algebra $\mathcal{A}$, introduced by Fomin--Zelevinsky \cite{FZ02}, is a commutative algebra with a set of distinguished generators called \textit{cluster variables}. Given an \textit{initial seed} $(\textbf{x},\tilde{B})$ where $\mathbf{x}$ refers to a tuple of $m+f$ number of \textit{initial cluster variables} and $\tilde{B}$ refers to an \textit{exchange matrix}, one uses the iterative combinatorial procedure called \textit{mutation} to produce the rest of the cluster variables, see Section \ref{sec: cluster alg} for more precise definitions.
The \textit{cluster automorphism group} of $\mathcal{A}$, denoted by $\text{Aut}(\mathcal{A})$, is the collection of algebra automorphisms that scales the cluster variables of $\mathcal{A}$, see Definition \ref{def: cluster auto}. We note that although cluster automorphism group and cluster modular group both consider automorphisms on $\mathcal{A}$, our definition of the cluster automorphism group is different from that of the cluster modular group \cite{hughes}. Cluster automorphism group was first studied by Gekhtman--Shapiro--Vainshtein in \cite{GSV} and later by Lam--Speyer in \cite{LS22}. Recently, understanding the action of cluster automorphism group became significant part in the study of \textit{cluster deep loci} and \textit{no mysterious point conjecture} in \cite{clusterloci}.

In this paper, we investigate the cluster automorphism group of braid varieties. Braid varieties are special class of affine algebraic varieties indexed by a pair of permutation $u\in S_n$ and a positive braid representative $\bm\beta=\sigma_{i_1}\cdots \sigma_{i_k}$, see Definition \ref{defn: braid var}. It is defined as a space of configurations of flags \[
  X_{u,\bm\beta}=\left\{(F_{1},F_{2},\ldots, F_{k})\in (G/B_{+})^{k}: B_{+}\xrightarrow{s_{i_1}} F_{1}\xrightarrow{s_{i_2}} F_{2}\to \ldots \xrightarrow{s_{i_k}}F_{k}\xleftarrow{w_0u} B_{-}\right\},
\]
where $G=GL(n,\mathbb{C})$ and $B_{+}$ (resp. $B_{-}$) denotes the Borel subgroup (resp. opposite Borel subgroup) of $G$. It generalizes many well-known Lie theoretic varieties, such as \textit{Grassmannian}, \textit{positroid varieties}, \textit{open Richardson varieties} and \textit{double Bott--Samelson varieties}, see \cite{CGGS,GL19, GY06b,scott,SSBW19}. In recent years, braid varieties have been a topic of interest due to their wide range of connection to \textit{Khovanov--Rozansky link homology} \cite{GL24kl} and \textit{exact Lagrangian fillings and Legendrian links} \cite{CGGLSS22}. 
The recent works of \cite{CGGLSS22} or \cite{GLSB23,GLSBS22} provide separate cluster structures in the coordinate ring of a braid variety.
These cluster structures are explained with combinatorial models. For example, \cite{GLSBS22} used a combinatorial model called the \textit{3D plabic graph} to explain a cluster structure given in the coordinate ring of a braid variety. In this paper, we mainly follow the cluster structure construction in \cite{GLSBS22}. From now on, we use 
$\widetilde{B}^{u,\bm\beta}$ to refer to the $\tilde{B}$ from an initial seed associated to $X_{u,\bf\beta}$ in \cite{GLSBS22}.

Matrix $\til$ exhibits many nice properties. For example, it is known that it is possible to extend $\widetilde{B}^{u,\bm\beta}$ to an invertible matrix, see Remark \ref{rmk: really full rank}. Once we have an invertible matrix extension of $\widetilde{B}^{u,\bm\beta}$, then \cite[Proposition 5.1]{LS22} implies that the action of $\text{Aut}(\mathcal{A})$ is described via the basis of the inverse of an extension matrix. Thus, in order to understand the action of a cluster automorphism group of braid varieties, we need to explicitly understand an invertible extension of $\widetilde{B}^{u,\beta}$ and its inverse. 

Now we are able to state our main result. Let us assume that $\widetilde{B}^{u,\beta}$ is a rectangular matrix of size ${m\times (m+f)}$ where $m$ and $f$ denotes the number of \textit{mutables} and \textit{frozen vertices} in the given $\mathbf{x}$ respectively. We define the {\bf square} $(m+f)\times (m+f)$ matrix $\bhat$ using the 3D plabic graph $G_{u,\bf \beta}$.

\begin{mainthm}
The matrix $\bhat$ is an integer matrix with
$\det \bhat=(-1)^{m+f}$. Thus, its inverse $\ahat$ is an integer matrix.
\end{mainthm}


As an application, we study the cluster automorphism group of braid varieties and its action using $\ahat$.
In Lemma \ref{lem: b tilde and ahat}, we show that the matrix $\til$ is the free abelian group of rank $f$ generated by vectors $({col}_j(\ahat))_{m+1\le j\le m+f}$ where ${col}_j(\ahat)$ denotes the $j$th column vector of $\ahat$.
Combined with Theorem \ref{thm: lam-speyer},
therefore, the action of the cluster automorphism group is explicitly described
as follows.

\begin{cor}\label{cor: main cor}
For $1\le j\le f$,
    suppose that ${col}_{m+j}(\ahat)=(a_{1,m+j},\ldots,a_{m+f,m+j})$.
    Then $\mathrm{Aut}(\mathcal{A}(Q_{u,\bm\beta}))$ is an algebraic torus with coordinates $t_1,\ldots,t_f$ which acts on $\X$ by 
    \[
x_i\to \prod t_j^{a_{i,m+j}}\cdot x_i,
\]
where $\textbf{x}=(x_1,\ldots, x_{m+f})$ be an initial seed of $\X$.
\end{cor}

In Section \ref{sec: ex of last sec}, we explore many examples and conputations of $\ahat$ and apply Corollary \ref{cor: main cor}.
In many cases, we observe a curious phenomenon that all entries of $\ahat$ are non-positive. We call this the \textbf{sign phenomenon}. We list some of the supporting examples and counterexample of this phenomenon in Section \ref{sec: ex of last sec}. It would be interesting to understand the geometric meaning of this phenomenon. In line with this inquiry, we propose the following problem.

\begin{problem}
Give an explicit combinatorial description for the non-zero entries of $\ahat$.
\end{problem}

Given a positive $n$-strand braid word $\beta$, we have an explicit description of the action of $(\C^{\times})^{n-1}$ on braid varieties, see \cite{CGGS} for more details. This action is given by cluster automorphisms, thus we have a homomorphism $(\C^{\times})^{n-1}\to \text{Aut}(\mathcal{A}(Q_{u,\bm\beta}))$. This homomorphism is surjective for $\X$ when $u=id$, but not in general. In examples in Section \ref{sec: ex of last sec}, we see that $f$ could be larger than $n-1$, so the cluster automorphism group has higher dimension than $(\C^{\times})^{n-1}$ in general. This leads to the following question.

\begin{problem}
Describe the action of the cluster automorphism group $\text{Aut}(\mathcal{A}(Q_{u,\bm\beta}))\cong(\C^{\times})^f$ on the flags of the braid variety.
\end{problem}

\subsection{Relationship to other related works}

The construction of our matrix $\bhat$ is indeed motivated by another invertible extension of $\til$, which was defined in \cite[Section 8.2]{CGGLSS22}. The forthcoming works \cite{comparisonpaper,CKW} will imply that our matrix is essentially same to this another extension for the correct choice of the weave. 

However, the key new things are the followings. First, our matrix construction is adapted and translated to the language of 3D plabic graphs. Our matrix has advantage in computation in practice. For example, one can use Galashin's program \cite{Galashin} when computing lots of examples, see Section \ref{sec: ex of last sec}. Also, we give an \textbf{explicit inductive factorization} of the inverse of $\bhat$, see Lemma \ref{lem: inductive factorization for inverse}. Therefore, combined with Corollary \ref{cor: main cor}, this leads to an inductive description of the cluster automorphism group action on braid varieties.


\section*{Acknowledgments}
The author would like to thank Eugene Gorsky, Melissa Sherman-Bennett, David Speyer and  Jos\'e Simental and Amanda Schwartz for the useful discussions and helpful comments. The author is also deeply in debt to Pavel Galashin for his code in double braid varieties \cite{Galashin}, most braid diagrams and quivers are similar to his code. 
The author is partially supported by the NSF grant DMS-2302305 and DMS-2103582.
\section{Background}

\subsection{Braid varieties}
We fix some notations. Let $I=[n-1]$ where $n\in \mathbb{N}$. The symmetric group $S_{n}$ is generated by simple transpositions $s_{i}=(i\ i+1)$ for $1\le i\le n-1$. 
We denote by $w_0$ the longest element in $S_n,$ $w_0(i)=n+1-i$. 
\begin{defn}
    The braid group $\mathrm{Br}_{n}$ on $n$ strands is generated by Artin   generators $\sigma_i$ and braid relations 
\[
 \sigma_{i}\sigma_{i+1}\sigma_{i}=\sigma_{i+1}\sigma_{i}\sigma_{i+1}, \quad \sigma_{i}\sigma_{j}=\sigma_{j}\sigma_{i}.
\]
An element $\beta\in \mathrm{Br}_n$ is called a positive braid if it is a product of Artin generators, allowing no inverses. Denote $\Br$ to be the collection of all positive braids in $\mathrm{Br}_n$. 
\end{defn}

Consider specific representative ${\beta}=\sigma_{i_{1}}\sigma_{i_{2}}\cdots \sigma_{i_{k}}$ of positive braid $\beta$, abbreviated by $k$-tuple $\bm{\beta}=(i_1,i_2,\ldots, i_k)\in I^k$. 
Now let $G$ be a complex general linear group and $B_{+}$ (resp. $B_{-}$) be the group of upper (resp. lower) triangular matrices in $GL(n,\mathbb{C})$. An element $F_i\in G/B_{+}$ is called a flag. We say that two flags $F,F'\in G/B_{+}$ are in position $w\in S_n$ if there exists a matrix $g\in G$ such that $(gF,gF')=(B_{+},wB_{+})$. 

\begin{defn}\cite{CGGS, CGGLSS22, GLSBS22}\label{defn: braid var}
    With these definitions in mind, 
a braid variety associated to $u\in S_n$ and $\bm\beta\in I^k$ is defined as 

\[
  X_{u,\bm\beta}=\left\{(F_{1},F_{2},\ldots, F_{k})\in (G/B_{+})^{k}: B_{+}\xrightarrow{s_{i_1}} F_{1}\xrightarrow{s_{i_2}} F_{2}\to \ldots \xrightarrow{s_{i_k}}F_{k}\xleftarrow{w_0u} B_{-}\right\}.
\]
\end{defn}

\begin{remark}
The braid variety $\X$ is invariant under braid relations, up to some change of variables.
If it is nonempty, it is a smooth complex manifold of dimension
$\ell(\beta)+\ell(u^{-1}w_0)-\binom{n}{2}$. 

\end{remark}

 \begin{remark}
     Note that if $\bm\beta$ is a reduced word, then the corresponding braid variety $\X$ is an open Richardson variety.
 \end{remark}

\subsection{Cluster algebra}\label{sec: cluster alg}

We start by recalling the definition of a quiver, following \cite{LS22, lauren}. A quiver $Q$ is a directed graph with a finite vertex set, where we allow multiple arrows between vertices but do not allow directed cycles of length one or two. We declare that $Q$ is an ice quiver, meaning that the vertices $1,2,\ldots, m$ of $Q$ are designated as \textit{mutable} whereas the remaining vertices $m+1,\ldots, m+f$ in $Q$ are designated as \textit{frozen}. Now we fix $\mathcal{F}$ to be an \textit{ambient field} of rational functions in $m$ independent variables over $\mathbb{C}(x_{m+1},\ldots, x_{m+f})$. Then \textit{labeled seed} in $\mathcal{F}$ is defined to be a pair $(\textbf{x}, Q)$, where $\textbf{x}=(x_1,\ldots, x_{m+f})$ forms a free generating set for $\mathcal{F}$ and $Q$ is an ice quiver as before. Equivalently, labeled seed $(\textbf{x}, Q)$ is identified with a pair $(\textbf{x},\widetilde{B})$ where $\widetilde{B}=\widetilde{B}(Q)$ from Definition \ref{def: B tilde defn}.

\begin{defn}\cite{LS22}\label{def: B tilde defn}
Let us fix $Q$ as above. Then an \textit{extended exchange matrix} associated to $Q$ is defined by \begin{equation}\label{eq: b tilde}
    \widetilde{B}(Q)=\left(b_{i,j}\right)_{1\le i\le m+f, 1\le j \le m},
\end{equation}
where $$
b_{i,j}=\begin{cases}
			r, & \text{if there are $r$ arrows from $i$ to $j$}\\
            -r, & \text{if there are $r$ arrows from $j$ to $i$}.
		 \end{cases}
$$
\end{defn}

\begin{remark}
    Note that this matrix is transposed to what is usually considered in the literature, for example, as in \cite{LS22}.
\end{remark}

\begin{remark}
    Later in Definition \ref{def: half arrow quiver cit}, we construct a quiver $Q_{u,\bf\beta}$. As mentioned in the introduction, we use $\widetilde{B}^{u,\bf\beta}$ to denote $\widetilde{B}(Q_{u,\bf\beta})$ for simplicity.
\end{remark}

\begin{defn}[\cite{FZ02}]
Let $k$ be a mutable vertex in a quiver $Q$. Then quiver mutation $\mu_k$ transforms $Q$ into a new quiver $\mu_k(Q)$ via the following steps: \begin{enumerate}
        \item For every directed pairs $i\to k\to j$, introduce new arrow $i\to j$.
        \item Reverse all arrows adjacent to a vertex $k$.
        \item Remove any maximal collection of disjoint $2$-cycles.
    \end{enumerate}
\end{defn}



\begin{defn}\cite{FZ02}
Let $(\textbf{x},Q)$ be a labeled seed in $\mathcal{F}$ as before, and pick $k\in \{1,2,\ldots, m\}$. The \textit{seed mutation} $\mu_k$ is a local operation that transforms an labeled seed $(\textbf{x},Q)$ into another labeled seed $(\textbf{x}', \mu_k(Q)),$ where $\textbf{x}'=(x_1',\ldots, x_m')$ is defined as follows:
    \begin{itemize}
        \item $x_j'=x_j$ for all $j\neq k$.
        \item $x_k'\in \mathcal{F}$ is determined by the \textit{exchange relation} \[
        x_k'\cdot x_k=\underset{\text{$k\to i$ in $Q$}}{\Pi}x_{i}+\underset{\text{$j\to k$ in $Q$}}{\Pi}x_{j}.
        \]

    \end{itemize}
\end{defn}

 To obtain a cluster algebra from an initial labeled seed $(\textbf{x},Q)$, we need to apply mutations to every mutable vertices repeatedly to obtain all possible cluster variables. 
\begin{defn}(\cite{FZ02})
The cluster algebra $\mathcal{A}(Q)$ associated to $Q$ is the subalgebra of the ring $\C(\textbf{x})$
generated as a ring by cluster variables together with $\{x_i^{-1}\}_{m+1\le i\le m+f}$ where the set of (generally infinitely many) cluster variables is constructed from the data $(\textbf{x}, Q)$ using seed mutation above.

\end{defn}


Separate groups of authors (see \cite{CGGLSS22} and \cite{GLSBS22}) showed that the coordinate ring of a braid variety admits a cluster structure. Two distinct combinatorial models called the Demazure weaves and 3D plabic graphs provide initial seeds on braid varieties. In this paper, we focus on seeds from 3D plabic graphs, which we now recall.



\subsection{3D plabic graph}\label{sec: 3d}


We recall the construction of a 3D plabic graph, following \cite{GLSBS22}. Note that the original terminology of 3D plabic graph refers to a graph drawn in $\mathbb{R}^3$, though in this paper, 3D plabic graph refers to its ``red projection" on $\mathbb{R}^2$; see \cite[Definition 3.4]{GLSBS22}. 

\begin{defn}\label{demazure quotient}
    Suppose that $v\in S_n$ is a permutation. Then the \textit{Demazure quotient} associated to $v$ and $i_l$ is defined by \[
v\lhd s_{i_l}:=\begin{cases}
    vs_{i_l} &  \text{if $\ell(vs_{i_l})<\ell(v)$},\\
    v & \text{otherwise},
\end{cases}
\]
where $\ell(v)$ denotes the length of a reduced word of $v$. 
\end{defn}

\begin{defn}\label{defn: pds}
Let $\bm\beta=(i_1,\ldots, i_k)\in I^k$ be a braid word representative of $\beta$ and $u$ be a permutation. We define 
the sequence of permutations $u_j$ recursively by setting $u_k:=u$ and
\[
u_{j-1}:=u_{j}\lhd s_{i_j},
\]
where $1\le j\le k$. With this sequence of permutations $u_j$ in mind, let us define the set $J_{u,\bm\beta}:=\{l\in [k]: u_{l}=u_{l-1}\}$. We will ultimately produce a quiver with $|J_{u,\bm\beta}|$ vertices.
\end{defn}

One can prove (e.g. \cite{GLSBS22}) that the letters with indices not in $J_{u,\bm\beta}$ produce the rightmost subexpression, also called the positive distinguished subexpression for $u$ inside $\bm\beta$, see \cite{Deodhar}. 

\begin{remark}
   Note that $u$ is a subword of $\bm\beta$ if and only if $u_0=1$. Then the braid variety $\X$ is not empty if and only if $\bm{\beta}$ contains a reduced expression for $u$ as a subword.
\end{remark}

    Throughout this paper, we will always assume that $u$ is a subword of $\bm\beta$.

\begin{defn}\label{defn: 3d plabic}
Let $\bm{\beta}$ be a given specific braid word and $u\in S_n$ be a subword of $\beta$ as above. The 3D plabic graph $G_{u,\bm{\beta}}$ is a projection of a graph in $\mathbb{R}^3$ to $\mathbb{R}^2$, drawn by the following procedure:
\begin{enumerate}
    \item [(a)] Put $n$ strands starting from the leftmost border. For these $n-1$ regions on the leftmost border, we call them \newword{boundary regions}. We number them from bottom to top using $1,2,\ldots,n$.
    \item [(b)] Now we go from $j=1$ to $j=k$. For $j\in J_{u,\bm\beta}$, draw a \textit{bridge} $d$ between the $i_j$th strand and the $(i_j+1)$th strand. 
    See Figure (left) below. Otherwise, draw a \textit{positive crossing} between the $i_j$th strand and the $(i_j+1)$th strand, as in the Figure (right) below.     
\end{enumerate}
\end{defn}
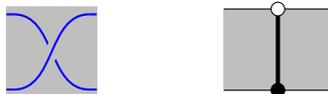
\begin{figure}[!h]
    \centering

    \begin{tikzpicture}[scale=1.2]
    \filldraw[color=lightgray] (-0.1,0.05) -- (-0.1,0.95) -- (1.1,0.95) -- (1.1,0.05) -- cycle;
    \draw (-0.1,0.05) -- (1.1,0.05);
    \draw (-0.1,0.95) -- (1.1,0.95);
    \draw [ultra thick] (0.5,0.95) to (0.5,0.05);
            \node[scale =0.5, circle, draw=black, fill=white] (W) at (0.5,0.95) {};
\node[scale =0.5, circle, draw=black, fill=black] (B) at (0.5,0.05) {};
    \end{tikzpicture}
    \qquad\qquad
    \begin{tikzpicture}
    \filldraw[color=lightgray] (-0.1,-0.1) -- (-0.1,1.1) -- (1.1,1.1) -- (1.1,-0.1) -- cycle;
      \draw[color=blue, thick] (-0.1,1) --(0,1);
  \draw[color=blue, thick] (1,1) --(1.1,1);
   \draw[color=blue, thick] (-0.1,0) --(0,0);
  \draw[color=blue, thick] (1,0) --(1.1,0);
\draw[color=blue, thick] (0,1) to [out=0,in=180] (1,0);
\draw[color=lightgray, line width=5] (0,0) to [out=0,in=180] (1,1);
\draw[color=blue, thick] (0,0) to [out=0,in=180] (1,1);
    \end{tikzpicture}
\end{figure}
\begin{remark}
    Braid words $\bm{\beta}$ and $\bm{\beta'}$ for the same braid $\beta$ produce different 3D plabic graphs. However, they are related by local moves. See \cite[Section 4]{GLSBS22}.
\end{remark}
\begin{exam}\label{ex: running ex intro}
    Let us consider the following running example in Figure \ref{fig: running ex}. Take $u=(1,2,5,4,3,6)=s_4s_3s_4\in S_6$ 
    and $\bm\beta=(5,4,3,2,1,4,3,\textcolor{blue}{4},2,5,\textcolor{blue}{3},\textcolor{blue}{4},5)$. The blue letters indicate the indices not in $J_{u,\beta}$, thus $J_{u,\bm\beta}=[13]\setminus \{8,11,12\}$. The resulting 3D plabic graph $G_{u,\bm\beta}$ is drawn in Figure \ref{fig: running ex}. Note that $G_{u,\bm\beta}$ has $5$ boundary regions on the left. 
\end{exam}

Next, we construct a soap film associated to each bridge in $G_{u,\bm\beta}$, which was introduced in \cite[Section 3.4]{GLSBS22} in a different terminology. 


\begin{defn}\label{defn: soap}
    Given a $G_{u,\bm\beta}$ and a bridge $d$ from $G_{u,\bm\beta}$, define the \newword{soap film} $C_{d}$ associated to a bridge $d$ as follows:
    \begin{enumerate}
        \item The soap film $C_{d}$ starts at a bridge $d$, and it continues to the left, clinging to the two horizontal strands that the bridge is attached to. 
        \item 
There are seven possibilities for how the soap film can appear immediately to the right of any bridge. Figure \ref{fig: propagation rule} shows how the soap film changes as it moves past the bridge in each of these seven scenarios. It is also possible that $C_d$ is above both horizontal strands or below both horizontal strands, In this case, nothing happens.
    \end{enumerate}
 The soap film is a region bounded by the edges in $G_{u,\beta}$, with additional information of whether it passes over or under the horizontal strands.
\end{defn}

    \begin{remark}
    To clarify, if a soap film encounters a positive crossing along its way, a soap film might become wider or smaller, see Figure \ref{fig:behavior at crossing}.
\end{remark}

\begin{figure}[hb!]
 \begin{tikzpicture}[scale=0.95]

\filldraw[color=piggypink] (-7,1) -- (-6,1) -- (-6,2) -- (-7,2) -- cycle;
\draw [thick] (-8,1) to (-6,1);
\draw [thick] (-8,2) to (-6,2);

 \draw [ultra thick] (-7,1) to (-7,2);
\node [scale =0.5, circle, draw=black, fill=black] (b1) at (-7,1) {};
\node [scale =0.5, circle, draw=black, fill=white] (w1) at (-7,2) {};
\draw [ultra thick, color=frenchrose] (-6.7,1.9) .. controls (-6.5,1.7) and (-6.5,1.3) .. (-6.7,1.1);
\draw [ultra thick, color=frenchrose] (-6.3,1.9) .. controls (-6.1,1.7) and (-6.1,1.3) .. (-6.3,1.1);

\draw[dashed, color=blue] (-5.8,2.5) to (-5.8,0.5) ;

\filldraw[color=piggypink] (-3.6,2) -- (-4.6,2) -- (-4.6,1) -- (-5.6,1) -- (-5.6,0.5)-- (-3.6,0.5) --(-3.6,2)-- cycle;
\draw [thick] (-5.6,1) to (-3.6,1);
\draw [thick] (-5.6,2) to (-3.6,2);

\draw [ultra thick] (-4.6,1) to (-4.6,2);
\node [scale =0.5, circle, draw=black, fill=black] (b2) at (-4.6,1) {};
\node [scale =0.5, circle, draw=black, fill=white] (w2) at (-4.6,2) {};

\draw [ultra thick, color=frenchrose] (-4.3,1.9) .. controls (-4.1,1.4) and (-4.1,1) .. (-4.3, 0.5);
\draw [ultra thick, color=frenchrose] (-3.9,1.9) .. controls (-3.7,1.4) and (-3.7,1) .. (-3.9, 0.5);
\draw [ultra thick, color=frenchrose] (-5,0.9)  .. controls (-4.9,0.8) and (-4.9,0.7) .. (-5,0.5);
\draw [ultra thick, color=frenchrose] (-5.4,0.9) .. controls (-5.3,0.8) and (-5.3,0.7) .. (-5.4, 0.5);
\draw[dashed, color=blue] (-3.4,2.5) to (-3.4,0.5) ;

\filldraw[color=piggypink] (-1.2,2.5) -- (-3.2,2.5) -- (-3.2,2) -- (-2.2,2) -- (-2.2,1) --(-1.2,1)--(-1.2,2.5)-- cycle;
\draw [thick] (-3.2,1) to (-1.2,1);
\draw [thick] (-3.2,2) to (-1.2,2);

\draw [ultra thick] (-2.2,1) to (-2.2,2);
\node [scale =0.5, circle, draw=black, fill=black] (b2) at (-2.2,1) {};
\node [scale =0.5, circle, draw=black, fill=white] (w2) at (-2.2,2) {};
\draw [ultra thick, color=frenchrose] (-2.5,2.5) .. controls (-2.4,2.4) and (-2.4,2.3) .. (-2.5, 2.1);
\draw [ultra thick, color=frenchrose] (-2.9,2.5) .. controls (-2.8,2.4) and (-2.8,2.3) .. (-2.9, 2.1);
\draw [ultra thick, color=frenchrose] (-1.9,2.5) .. controls (-1.8,2.4) and (-1.8,2.3) .. (-1.9, 2.1);
\draw [ultra thick, color=frenchrose] (-1.5,2.5) .. controls (-1.4,2.4) and (-1.4,2.3) .. (-1.5, 2.1);

\draw [ultra thick, color=frenchrose] (-1.9,1.9) .. controls (-1.7,1.7) and (-1.7,1.3) .. (-1.9, 1.1);
\draw [ultra thick, color=frenchrose] (-1.5,1.9) .. controls (-1.3,1.7) and (-1.3,1.3) .. (-1.5, 1.1);

\draw[dashed, color=blue] (-1,2.5) to (-1,0.5) ;

\filldraw[color=piggypink] (1.2,2.5) -- (-0.8,2.5) -- (-0.8,2) -- (0.2,2) --(0.2,1)--(-0.8,1)--(-0.8,0.5)--(1.2,0.5)-- 
cycle;
\draw [thick] (-0.8,1) to (1.2,1);
\draw [thick]  (-0.8,2) to (1.2,2);

\draw [ultra thick] (0.2,1) to (0.2,2);
\node [scale =0.5, circle, draw=black, fill=black] (b2) at (0.2,1) {};
\node [scale =0.5, circle, draw=black, fill=white] (w2) at (0.2,2) {};

\draw [ultra thick, color=frenchrose] (-0.5,2.5) .. controls (-0.4,2.4) and (-0.4,2.3) .. (-0.5, 2.1);
\draw [ultra thick, color=frenchrose] (-0.1,2.5) .. controls (0,2.4) and (0,2.3) .. (-0.1, 2.1);


\draw [ultra thick, color=frenchrose] (-0.5,0.9) .. controls (-0.4,0.8) and (-0.4,0.7) .. (-0.5, 0.5);
\draw [ultra thick, color=frenchrose] (-0.1,0.9) .. controls (0,0.8) and (0,0.7) .. (-0.1, 0.5);


\draw [ultra thick, color=frenchrose] (0.5,2.5) .. controls (0.6,2.4) and (0.6,2.3) .. (0.5, 2.1);
\draw [ultra thick, color=frenchrose] (0.9,2.5) .. controls (1,2.4) and (1,2.3) .. (0.9, 2.1);


\draw [ultra thick, color=frenchrose] (0.5,1.9) .. controls (0.7,1.4) and (0.7,1) .. (0.5, 0.5);
\draw [ultra thick, color=frenchrose] (0.9,1.9) .. controls (1.1,1.4) and (1.1,1) .. (0.9, 0.5);

\draw[dashed, color=blue] (1.4,2.5) to (1.4,0.5) ;

\filldraw[color=piggypink] (1.6,0.5) -- (3.6,0.5) -- (3.6,2) -- (1.6,2) -- cycle;
\draw [thick] (1.6,1) to (3.6,1);
\draw [thick]  (1.6,2) to (3.6,2);

\draw [ultra thick] (2.6,1) to (2.6,2);
\node [scale =0.5, circle, draw=black, fill=black] (b2) at (2.6,1) {};
\node [scale =0.5, circle, draw=black, fill=white] (w2) at (2.6,2) {};

\draw [ultra thick, color=frenchrose] (3.3,1.9) .. controls (3.5,1.7) and (3.5,1.3) .. (3.3,1.1);
\draw [ultra thick, color=frenchrose] (2.9,1.9) .. controls (3.1,1.7) and (3.1,1.3) .. (2.9,1.1);


\draw [ultra thick, color=frenchrose] (3.3,0.9) .. controls (3.4,0.8) and (3.4,0.7) .. (3.3, 0.5);
\draw [ultra thick, color=frenchrose] (2.9,0.9) .. controls (3,0.8) and (3,0.7) .. (2.9, 0.5);


\draw [ultra thick, color=frenchrose] (2.3,1.9) .. controls (2.5,1.7) and (2.5,1.3) .. (2.3,1.1);
\draw [ultra thick, color=frenchrose] (1.9,1.9) .. controls (2.1,1.7) and (2.1,1.3) .. (1.9,1.1);


\draw [ultra thick, color=frenchrose] (2.3,0.9) .. controls (2.4,0.8) and (2.4,0.7) .. (2.3, 0.5);
\draw [ultra thick, color=frenchrose] (1.9,0.9) .. controls (2.0,0.8) and (2.0,0.7) .. (1.9, 0.5);

\draw[dashed, color=blue] (3.8,2.5) to (3.8,0.5) ;

\filldraw[color=piggypink] (4,1) -- (6,1) -- (6,2.5) -- (4,2.5) -- cycle;
\draw [thick]  (4,1) to (6,1);
\draw [thick]  (4,2) to (6,2);

\draw [ultra thick] (5,1) to (5,2);
\node [scale =0.5, circle, draw=black, fill=black] (b2) at (5,1) {};
\node [scale =0.5, circle, draw=black, fill=white] (w2) at (5,2) {};

\draw [ultra thick, color=frenchrose] (5.7,2.5) .. controls (5.9,2) and (5.9,1.6) .. (5.7,1.1);
\draw [ultra thick, color=frenchrose] (5.3,2.5) .. controls (5.5,2) and (5.5,1.6) .. (5.3,1.1);


\draw [ultra thick, color=frenchrose] (4.6,2.5) .. controls (4.8,2) and (4.8,1.6) .. (4.6,1.1);
\draw [ultra thick, color=frenchrose] (4.2,2.5) .. controls (4.4,2) and (4.4,1.6) .. (4.2,1.1);

\draw[dashed, color=blue] (6.2,2.5) to (6.2,0.5) ;

\filldraw[color=piggypink] (6.4,0.5) -- (8.4,0.5) -- (8.4,2.5) -- (6.4,2.5) -- cycle;
\draw [thick]  (6.4,1) to (8.4,1);
\draw [thick]  (6.4,2) to (8.4,2);

\draw [ultra thick] (7.4,1) to (7.4,2);
\node [scale =0.5, circle, draw=black, fill=black] (b2) at (7.4,1) {};
\node [scale =0.5, circle, draw=black, fill=white] (w2) at (7.4,2) {};

\draw [ultra thick, color=frenchrose] (7.7,2.5) .. controls (7.9,2) and (7.9,1.6) .. (7.7,1.1);
\draw [ultra thick, color=frenchrose] (8.1,2.5) .. controls (8.3,2) and (8.3,1.6) .. (8.1,1.1);


\draw [ultra thick, color=frenchrose] (7.7,0.9) .. controls (7.8,0.8) and (7.8,0.7) .. (7.7, 0.5);
\draw [ultra thick, color=frenchrose] (8.1,0.9) .. controls (8.2,0.8) and (8.2,0.7) .. (8.1, 0.5);


\draw [ultra thick, color=frenchrose] (6.6,2.5) .. controls (6.8,2) and (6.8,1.6) .. (6.6,1.1);
\draw [ultra thick, color=frenchrose] (7,2.5) .. controls (7.2,2) and (7.2,1.6) .. (7,1.1);


\draw [ultra thick, color=frenchrose] (6.6,0.9) .. controls (6.7,0.8) and (6.7,0.7) .. (6.6, 0.5);
\draw [ultra thick, color=frenchrose] (7,0.9) .. controls (7.1,0.8) and (7.1,0.7) .. (7, 0.5);

    \end{tikzpicture}
    \caption{ 
A soap film can be thought of as \textcolor{frenchrose}{light pink} region which can go over or under the horizontal strands of the graph $G_{u,\beta}$. The \textcolor{magenta}{darker pink} color indicates whether a soap film goes over or under.
In short, if the soap film is going over the top strand, or under the bottom strands when it reaches the bridge, then the soap film does not change when it passes the bridge. Otherwise, the bridge is part of the boundary of the soap film.
}
    \label{fig: propagation rule}
\end{figure}

\begin{figure}
    \centering
   \begin{tikzpicture}

\filldraw[color=piggypink] (2.5,2)--(0,2)--(0,0)--(1,0)--(2,1)--(2.5,1) -- cycle;
\draw[thick] (0,0) to (1,0);
\draw[thick] (2,0) to (3,0);

\draw[thick] (0,1) to (1,1);
\draw[thick] (2,1) to (3,1);

\draw[thick] (0,2) to (3,2);
\draw[ultra thick] (2.5,2) to (2.5,1);
\node [scale =0.5, circle, draw=black, fill=black] (b1) at (2.5,1) {};
\node [scale =0.5, circle, draw=black, fill=white] (w1) at (2.5,2) {};
\draw[color=blue, thick] (1,0) to [out=0,in=180] (2,1);
\draw[color=blue,thick] (1,1) to [out=0,in=180] (2,0);
\draw [ultra thick, color=frenchrose] (2.2,1.9) .. controls (2.0,1.7) and (2.0,1.3) .. (2.2,1.1);
 \draw [ultra thick, color=frenchrose] (1.5,1.9) .. controls (1.3,1.7) and (1.3,1.1) .. (1.5,0.6);
 \draw [ultra thick, color=frenchrose] (0.7,1.9) .. controls (0.5,1.3) and (0.5,0.7) .. (0.7,0.1);



   \end{tikzpicture}
    \caption{Behavior of soap film propagation when it encounters a crossing}
    \label{fig:behavior at crossing}
\end{figure}


The soap films are in bijection with the vertices in a quiver.
If the resulting soap film $C_d$ is a closed and bounded region, the bridge $d$ corresponds to a mutable vertex. Otherwise, $C_{d}$ corresponds to a frozen vertex. We therefore specify that in such case, a bridge $d$ is a \newword{frozen bridge}, to emphasize its relationship with an initial quiver. 
\begin{remark}\label{rmk: leftmost frozen}
Followed by the condition (2) of Definition \ref{defn: soap}, the leftmost bridge is always a frozen bridge. For future usage in Section \ref{subsec: auxiliary lemmas B} and Theorem \ref{thm: main thm}, let us call this leftmost frozen bridge by $d_L$.
\end{remark}
Let us denote by $\mathcal{C}:=\{C_d\}_{d\in G_{u,\bm\beta}}$ the set of soap films. For future usage, we now give a ordering on the set $C$ as follows.

\begin{defn}\label{defn: total ordering on C}
Suppose that $d_a,d_b$ are bridges chosen from the 3D plabic graph $G_{u,\bm \beta}$. We define an ordering $<$ on $\mathcal{C}$ satisfying the following conditions:
\begin{enumerate}
    \item If $d_a$ and $d_b$ are bridges of the same type (either both frozen or mutable) and $d_a$ appears on the right of $d_b$ in $G_{u,\bm\beta}$, then $C_{a}>C_b$.
    \item If $d_a$ is mutable and $d_b$ is frozen, then $C_a<C_b$. 
\end{enumerate}
\end{defn}

\begin{remark}\label{rmk: total ordering on C}
The above ordering $<$ is a total ordering on $\mathcal{C}$, thus one may regard $\mathcal{C}$ as an ordered set $\{C_1<\cdots <C_{m+f}\}$, where soap films $C_1,\ldots, C_m$ corresponds to mutable vertices and $C_{m+1},\ldots, C_{m+f}$ corresponds to frozen vertices in $Q_{u,\bm\beta}$. From now on, we relabel the vertices so that now a vertex $j$ in $Q_{u,\bm\beta}$ is in correspondence to the soap film $C_j$ of $Q_{u,\bm\beta}$. Note that this ordering is different from the one used in \cite{GLSBS22}, where soap films are labeled by $J_{u,\bm{\beta}}$.
\end{remark}

\begin{exam}\label{ex: leftmost frozen soap}
Remark \ref{rmk: leftmost frozen} implies that $d_L$ always corresponds to $C_{m+1}$.
\end{exam}


\begin{exam}
    Continuing with Example \ref{ex: running ex intro}, one has $10$ soap films. For example, one can check from Figure \ref{fig: running ex} that the soap film $C_{{4}}$ is closed and bounded, whereas the soap film
    $C_{{10}}$ encompass the leftmost boundary, thus it is unbounded. 
    In total, there are $6$ frozen bridges. See Figure \ref{fig: running ex}.
\end{exam}

Turning our attention to each region in $G_{u,\bm\beta}$, i.e., each connected component of $G_{u,\bm\beta}$, we obtain an initial quiver $Q_{u,\beta}$ from $G_{u,\bm\beta}$. We use the \textit{half arrow configuration} in Figure \ref{fig:half arrow config}.

\begin{figure}
    \centering
   \begin{tikzpicture}
 \filldraw[color=piggypink] (1.5,0)  -- (4.5,0) -- (4.5,3) -- (1.5,3) -- cycle;
 \draw [thick] (1.5,0) to (4.5,0);
 \draw [thick] (1.5,1) to (4.5,1);
 \draw [thick] (1.5,2) to (4.5,2);
 \draw [thick] (1.5,3) to (4.5,3); 
\draw [ultra thick](3,1) to (3,2);
 \node [scale =0.5, circle, draw=black, fill=black] (b) at (3,1) {};
\node [scale =0.5, circle, draw=black, fill=white] (w) at (3,2) {};

 \draw[color=red, {Stealth[red]}-,thick] (3.5,1.7)-- (2.5,1.7)  ;
\draw[color=red, {Stealth[red]}-,thick] (3.5,1.3) -- (2.5,1.3) ;
\draw [color=red, {Stealth[red]}-,thick] (3.1, 2.3)-- (3.75, 1.7) ;
\draw [color=red, {Stealth[red]}-,thick] (2.25, 1.7)-- (2.9, 2.3) ;

\draw [color=red, {Stealth[red]}-,thick] (2.25, 1.2)--(2.9, 0.7);
\draw [color=red, {Stealth[red]}-,thick] (3.1, 0.7)--(3.75, 1.2);
 \end{tikzpicture}
    \caption{Half arrow configuration near a bridge. Each of six \textcolor{red}{red} arrows are half arrows, the name comes from the fact that their weights are $\frac{1}{2}$ when constructing the matrix $\Hmatrix$.}
    \label{fig:half arrow config}
\end{figure}

\begin{defn}\cite[Section 3.7]{GLSBS22}\label{def: half arrow quiver cit}
We construct an initial quiver $Q_{u,\bm\beta}$ from $G_{u,\bm\beta}$ via the following procedure.
\begin{enumerate}
    \item [(a)] 
    Around each bridge, place half arrows as
    in Figure \ref{fig:half arrow config}. There are no half arrows near crossings. 
    \item [(b)] Label $R$ with the set $\mathcal{R}:=\{i: \text{$C_i$ covers $R$}\}$.
    If a region has no numbers written, we just ignore it. Repeat this process for every soap film $C_{i}$ of $G_{u,\bm\beta}$.
    \item [(c)] For each half arrow $\mathcal{R}\to \mathcal{R'}$, draw one half arrow $i\to j$ for each $(i,j)\in \mathcal{R}\times \mathcal{R'}$.
    
    \item [(e)] Delete a maximal collection of 2-cycles and loops. If there are $m$ half arrows from $i$ to $j$, treat this as $\frac{m}{2}$ arrows.
\end{enumerate}
\end{defn}

\begin{exam}
If we follow (c) in Definition \ref{def: half arrow quiver cit} around the leftmost bridge $d_L$ in Figure \ref{fig: running ex}, then we obtain the half arrows $\textcolor{red}{5}\to \textcolor{red}{10}$, $\textcolor{red}{5}\to \textcolor{red}{10}$, $\textcolor{red}{10}\to \textcolor{red}{10}$, $\textcolor{red}{10}\to \textcolor{red}{6}$, $\textcolor{red}{10}\to \textcolor{red}{5}$ and $\textcolor{red}{6}\to \textcolor{red}{5}$. 
\end{exam}

\begin{figure*}[ht!]
    \centering
\begin{tikzpicture}[scale=1]
\filldraw[color=piggypink] (-4.5,1) -- (-2,1) -- (-2,2) -- (-4.5,2) -- cycle;
\filldraw[color=almond](-6.5,3)--(-2.8,3) to [out=0,in=180] (-2.2,4) to (-1.5,4) to (-1.5,5) to (-6,5) to (-6,4) to (-6.5,4); 

\draw [thick](-6.5,0) -- (0,0);
\draw [thick](-6.5,1) -- (0,1);
\draw [thick](-6.5,2) -- (0,2);


\draw [thick](-6.5,3) -- (-2.8,3);
\draw [thick](-6.5,4) -- (-2.8,4);
\draw [thick](-2.2,3) -- (0,3);
\draw [thick](-2.2,4) -- (0,4);

\draw [thick](-6.5,5) -- (0,5);

\draw [ultra thick] (-6,4.95) to (-6,4.05);
\node[scale =0.5, circle, draw=black, fill=white] (W) at (-6,4.95) {};
\node[scale =0.5, circle, draw=black, fill=black] (B) at (-6,4.05) {};
\node [color=black] (B) at (-6,5.2) {5};

\node[color=frenchrose] (1) at (-6.25,4.5) {5};
\node[color=frenchrose] (2,10) at (-6,3.5) {6,10};
\node[color=frenchrose] (1) at (-6.25,2.5) {7};
\node[color=frenchrose] (1) at (-6.25,1.5) {8};
\node[color=frenchrose] (1) at (-6.25,0.5) {9};


\draw [ultra thick] (-5.5,3.95) to (-5.5,3.05);
\node[scale =0.5, circle, draw=black, fill=white] (W) at (-5.5,3.95) {};
\node[scale =0.5, circle, draw=black, fill=black] (B) at (-5.5,3.05) {};
\node[color=black] (W) at (-5.5,4.2) {6};
\draw [ultra thick] (-5,2.95) to (-5,2.05);
\node[scale =0.5, circle, draw=black, fill=white] (W) at (-5,2.95) {};
\node[scale =0.5, circle, draw=black, fill=black] (B) at (-5,2.05) {};
\node[color=black] (W) at (-5,3.2) {7};
\draw [ultra thick] (-4.5,1.95) to (-4.5,1.05);
\node[scale =0.5, circle, draw=black, fill=white] (W) at (-4.5,1.95) {};
\node[scale =0.5, circle, draw=black, fill=black] (B) at (-4.5,1.05) {};
\node[color=black] (W) at (-4.5,2.2) {8};
\draw [ultra thick] (-4,0.95) to (-4,0.05);
\node[scale =0.5, circle, draw=black, fill=white] (W) at (-4,0.95) {};
\node[scale =0.5, circle, draw=black, fill=black] (B) at (-4,0.05) {};
\node[color=black] (W) at (-4,1.2) {9};

\draw [ultra thick] (-3.5,3.95) to (-3.5,3.05);
\node[scale =0.5, circle, draw=black, fill=white] (W) at (-3.5,3.95) {};
\node[scale =0.5, circle, draw=black, fill=black] (B) at (-3.5,3.05) {};
\node[color=frenchrose] (1) at (-4.5,3.5) {1,10};
\node[color=black] (W) at (-3.5,4.2) {1};
\draw [ultra thick] (-3,2.95) to (-3,2.05);
\node[scale =0.5, circle, draw=black, fill=white] (W) at (-3,2.95) {};
\node[scale =0.5, circle, draw=black, fill=black] (B) at (-3,2.05) {};
\node[color=frenchrose] (1) at (-4.2,2.5) {2};

\draw[color=blue, thick] (-2.8,4) to [out=0,in=180] (-2.2,3);
\draw[color=blue, thick] (-2.8,3) to [out=0,in=180] (-2.2,4);
\node[color=black] (W) at (-3.2,2.8) {2};


\draw [ultra thick] (-2,1.95) to (-2,1.05);
\node[scale =0.5, circle, draw=black, fill=white] (W) at (-2,1.95) {};
\node[scale =0.5, circle, draw=black, fill=black] (B) at (-2,1.05) {};
\node[color=frenchrose] (1) at (-3,1.5) {3};
\node[color=black] (W) at (-2,2.2) {3};

\draw [ultra thick] (-1.5,4.95) to (-1.5,4.05);
\node[scale =0.5, circle, draw=black, fill=white] (W) at (-1.5,4.95) {};
\node[scale =0.5, circle, draw=black, fill=black] (B) at (-1.5,4.05) {};
\node[color=frenchrose] (1) at (-3,4.5) {10};
\node[color=frenchrose] (1) at (-3,3.5) {10};
\node[color=black] (W) at (-1.5,5.2) {10};



 \draw[color=white, line width=5] (-1.3,3) to (-0.7,3);
 \draw[color=white, line width=5] (-1.3,2) to (-0.7,2);
\draw[color=blue, thick] (-1.3,3) to [out=0,in=180] (-0.7,2);
\draw[color=white, line width=5] (-1.3,2) to [out=0,in=180] (-0.7,3);
\draw[color=blue, thick] (-1.3,2) to [out=0,in=180] (-0.7,3);

 \draw[color=white, line width=5] (-0.7,4) to (-0.2,4);
 \draw[color=white, line width=5] (-0.7,3) to (-0.2,3);
\draw[color=blue, thick] (-0.7,4) to [out=0,in=180] (-0.2,3);
\draw[color=white, line width=5] (-0.7,3) to [out=0,in=180] (-0.2,4);
\draw[color=blue, thick] (-0.7,3) to [out=0,in=180] (-0.2,4);

\draw [ultra thick] (0,4.95) to (0,4.05);
\node[scale =0.5, circle, draw=black, fill=white] (W) at (0,4.95) {};
\node[scale =0.5, circle, draw=black, fill=black] (B) at (0,4.05) {};
\node[color=black] (W) at (0,5.2) {4};
\node[color=frenchrose] (1) at (-0.5,4.5) {4};
\node[color=frenchrose] (1) at (-1.5,3.5) {4};
\node[color=frenchrose] (1) at (-2,2.5) {4};
\end{tikzpicture}
\caption{Soap films imposed on a 3D plabic graph $G_{u,\bm{\beta}}$ associated to a permutation $u=s_4s_3s_4$ and $\bm\beta=(5,4,3,2,1,4,3,\textcolor{blue}{4},2,5,\textcolor{blue}{3},\textcolor{blue}{4},5)$. A number $i$ around a bridge indicates its correspondence to the vertex $i$, following the convention in Remark \ref{rmk: ordering rmk}.
For example, the soap film $C_{10}$, indicated by the {brown} region, is the unbounded region with a number {frenchrose}{10}. One can see that a number \textcolor{frenchrose}{10} propagates to the leftmost boundary of $G_{u,\bm\beta}$ and thus the vertex $10$ is frozen, thus the vertex {$10$} is {frozen}. Similarly, soap film $C_{3}$, colored with a {pink} color, is closed and bounded, thus the vertex {$3$} is {mutable}.} \label{fig: running ex}
\end{figure*}


\begin{thm}\cite[Section 3.7]{GLSBS22}\label{thm: braid variety cluster}
The cluster algebra $\mathcal{A}(Q_{u,\bm\beta})$ from an initial quiver $Q_{u,\bm\beta}$ defined above is isomorphic to the coordinate ring of $X_{u,\bm\beta}$.
\end{thm}

\begin{remark}\label{rmk: arrow convention is different}
Note that our half arrow configuration in Figure \ref{fig:half arrow config} is the opposite of the half arrow convention in \cite[Section 3.7]{GLSBS22}. However, this does not change the cluster algebra.
\end{remark}

\begin{remark}
\label{rmk: arrows frozen}
As in \cite{CGGLSS22}, the quiver $Q_{u,\bm\beta}$ may contain (half) arrows between the frozen vertices. 
However, the number of arrows between the mutable vertices, or the number of arrows between a mutable vertex and a frozen vertex is an integer. 
\end{remark}


\begin{figure}
    \centering
   \begin{tikzpicture}[squarednode/.style={rectangle, draw=blue!60, fill=blue!5, very thick, minimum size=5mm},
]]
\node[draw,squarednode] (1) at (0,5){5};
\node[draw, squarednode] (2) at (0.5,4){6};
\node[draw, squarednode] (3) at (1,3){7};
\node[draw, squarednode] (4) at (1.5,2){8};
\node[draw, squarednode] (5) at (2,1){9};
\node[draw, circle, red] (6) at (2.5,4){1};
\node[draw, circle, red] (7) at (3,3){2};
\node[draw, circle, red] (9) at (4,2){3};
\node[draw, squarednode] (10) at (4.5,5){10};
\node[draw, circle, red] (13) at (6,5){4};
\draw[->] (2) to (6);
\draw[->] (6) to (3);
\draw[->] (3) to (7);
\draw[->] (7) to (4);
\draw[->] (4) to (9);
\draw[->] (9) to (5);
\draw[->] (7) to (6);
\draw[->] (9) to (7);
\draw[->] (7) to (13);
\draw[->] (13) to (9);  
\draw[->, blue!75] (1) to (10);  
\draw[->, blue!75] (5) to (4);  
\draw[->, blue!75] (4) to (3);  
\draw[->, blue!75] (3) to (2);
\draw[->, blue!75] (2) to (1);  
\end{tikzpicture}
    \caption{The quiver $Q_{u,\beta}$ from Figure \ref{fig: running ex} following from $G_{u,\beta}$ and Definition \ref{def: half arrow quiver cit}. The green color indicates that the corresponding vertex is mutable, whereas the red color indicates that the corresponding vertex is frozen. The number $j$ written in a vertex indicates its correspondence with a soap film $C_j$, defined in Remark \ref{rmk: total ordering on C}. The red arrows are half arrows as appeared in Figure \ref{fig:half arrow config}.}
    \label{fig:running ex quiver}
\end{figure}

%




\begin{defn}\label{defn: H matrix}
The \newword{half weight matrix} $\Hmatrix$ is defined by 
\[
\Hmatrix=\left(b_{i,j}\right)_{1\le i\le m+f, 1\le j \le m+f},
\]
where $$
b_{i,j}=\begin{cases}
			\frac{1}{2}r, & \text{if there are $r$ half arrows from $i$ to $j$}\\
            -\frac{1}{2}r, & \text{if there are $r$ half arrows from $j$ to $i$}.
		 \end{cases}
$$

\end{defn}

 Note that here we include the half-arrows between frozen vertices, compare with Remark \ref{rmk: arrows frozen}. If both $i$ and $j$ are frozen then $b_{i,j}$ is a half-integer or an integer.

\begin{remark}
    Observe that $\Hmatrix$ is always a skew symmetric square matrix. Note also that we recover the matrix $\til$ by only considering the first $m$ columns of $\Hmatrix$.
\end{remark}

\section{Main theorem and proof}



\subsection{Main theorem}\label{sec: rank}

In this subsection, we define the $\bhat$ matrix, which is an invertible extension of $\til$.
We recall the indexing convention in Remark \ref{rmk: ordering rmk}. 

\begin{defn}\label{defn: boundary}
Define the \newword{boundary map} $\p_{u,\bm\beta}: C\to \Z^{n-1}$ by 
\[
\p_{u,\bm\beta}(C_{i})=\underset{\text{$j$: $C_i$ covers $j$th boundary region of $G_{u,\bm\beta}$}}{\sum}\alpha_j
\]
and the \newword{boundary correction matrix}
\[
D_{i,j}^{u,\bm\beta}=\frac{1}{2}\p_{u,\bm\beta}(C_i)\cdot \p_{u,\bm\beta}(C_j)
\]
defined by the usual dot product. We set $D^{u,\bm\beta}:=(D_{i,j}^{u,\beta})_{1\le i,j\le m+f}$.
    
\end{defn}

\begin{remark}
Throughout all examples where we are given with the choice of $u$ and $\bm\beta$ and have no ambiguity, we will often abbreviate $\p_{u,\bm\beta}$ by $\p$.
\end{remark}

\begin{defn}\label{defn: bhat}
Here and throughout, we put 
$\bhat:=\Hmatrix+\D=(\widehat{b}_{i,j})_{1\le i,j\le m+f}$.
\end{defn}

\begin{remark}\label{rem: D is sym}
    Note that $D_{ij}=D_{ji}$, so that $D$ is a symmetric matrix of size $(m+f)\times (m+f)$.
\end{remark}


Note that
\[
   \Hmatrix=\leftidx{_{f}^{\phantom{m}\llap{$\scriptstyle m$}}}{\left( \begin{array}{c:c}
    \smash{\overset{m}{B}} & \mathclap{\smash{\overset{f}{E}}} \\
    \hdashline
    -E^T & N
  \end{array} \right)}{}
\]

and \[
   \D=\leftidx{_{f}^{\phantom{m}\llap{$\scriptstyle m$}}}{\left( \begin{array}{c:c}
    \smash{\overset{m}{0}} & \mathclap{\smash{\overset{f}{0}}} \\
    \hdashline
    0 & N'
  \end{array} \right)}{},
\]
where $N,N'$ are some $f\times f$ matrices. Thus we have

\begin{equation}
 \bhat=\leftidx{_{f}^{\phantom{m}\llap{$\scriptstyle m$}}}{\left( \begin{array}{c:c}
    \smash{\overset{m}{B}} & \mathclap{\smash{\overset{f}{E}}} \\
    \hdashline
    -E^T & N+N'
  \end{array} \right)}{}.
\end{equation}

\begin{remark}\label{rmk: ordering rmk}
    Note that $\D$ and $\bhat$ follow the ordering convention in Definition \ref{defn: H matrix}.
\end{remark}


\begin{defn}
We define two matrices $L=(\ell_{yx})$ and $R=(r_{yx})$ as follows:
$$
l_{yx}=\begin{cases}
    1 & y=x\\
    \widehat{b}_{y,m+1} & x=m+1\ \mathrm{and}\ y\neq m+1\\
    0 & \mathrm{otherwise}
\end{cases}\quad \text{and} \quad
r_{yx}=\begin{cases}
    1 & y=x\\
    \widehat{b}_{m+1,x} & y=m+1\ \mathrm{and}\ x\neq m+1\\
    0 & \mathrm{otherwise}.
\end{cases}
$$
Define $Z^{u,\bm\beta}:=\bhat R$.
\end{defn}


\begin{defn}\label{defn: Z hat}
Define another matrix $\widehat{Z}^{u,\bm\beta}$ obtained from $Z^{u,\bm\beta}$ by replacing all $\widehat{b}_{y,m+1}$ by $0$. 
\end{defn}
Observe that $\widehat{Z}^{u,\bm\beta}$ agrees with $\widehat{B}^{u,\bm\beta'}$ up to adding $(m+1)$-st row and column with $(-1)$ on diagonal and zeroes elsewhere and conjugation by a permutation matrix.  

\begin{lem}
\label{lem: L and R}
We have $L\widehat{B^{u,\bm\beta}}R=\widehat{Z}^{u,\bm\beta}$.
\end{lem}

\begin{proof}
Recall that $Z^{u,\bm\beta}=\widehat{B^{u,\bm\beta}}R$.
We can get from $Z^{u,\bm\beta}$ to $\widehat{Z}^{u,\bm\beta}$ by adding to $y$-th row the $(m+1)$-st row multiplied by $\widehat{b}_{y,m+1}$, this is equivalent to left multiplication by $L$, so $LZ^{u,\bm\beta}=\widehat{Z}^{u,\bm\beta}$ and the result follows. 
\end{proof}

The following is the main result of the paper. We prove it modulo a series of technical lemmas which are postponed to Section \ref{subsec: auxiliary lemmas A} and Section \ref{subsec: auxiliary lemmas B}.

\begin{thm}\label{thm: main thm}
For any $u$ and a specific braid representative $\bm\beta$, the matrix ${\bhat}$ has integer coefficients and $\det(\bhat)=(-1)^{m+f}$.
\end{thm}
\begin{proof}
Let us fix the specific braid word $\bm{\beta}$ of $\beta$.
We proceed on induction on length of $\bm\beta$. If $\bm{\beta}$ is a reduced word for $u$, then $Q_{u,\bm\beta}$ is empty and there is nothing to prove. Given $u\in S_n$ and $\bm{\beta}$, we have either two cases: $1\not\in J_{u,{\bm\beta}}$ or $1\in J_{u,{\bm\beta}}$. 
Let us first consider the first case, the \textbf{Case A}. The 3D plabic graph for $(u,{\bm \beta})$ is obtained from $G_{u',\bm\beta'}$
 by adding a crossing on the left. In particular, $G_{u,\bm\beta}$ and $G_{u',\bm\beta'}$ have the same number of bridges, and the same number of mutable and frozen variables.
By the induction hypothesis, we have $\det(\widehat{B^{u',\bm\beta'}})=\pm1$.
We have $H_{u,\bm\beta}=H_{u,\bm\beta'}$ since the quivers $Q_{u,\bm\beta}$ and $Q_{u',\bm{\beta'}}$ are the same, and we can identify respective frozen and mutable vertices. Now from Corollary \ref{cor: main cor for case a}, one has $D^{u,\bm\beta}=D^{u',\bm\beta'}$ and thus $\bhat=H^{u,\bm\beta}+D^{u,\bm\beta}=\widehat{B^{u',\bm{\beta'}}}$.

Therefore, in this case, we conclude that $\bhat$ is an integer matrix and $\det(\widehat{B^{u,\bm\beta}})=\det(\widehat{B^{u',\bm{\beta'}}})=(-1)^{m+f}$.
 
 Next, we consider the latter case, \textbf{Case B}. As in Remark \ref{rmk: leftmost frozen}, $1\in J_{u,\underline{\beta}}$ corresponds to the leftmost frozen bridge $d_L$. Comparing the quivers $Q_{u,\bm\beta}$ and $Q_{u',\bm{\beta'}}$, we notice from Example \ref{ex: leftmost frozen soap} that there is an extra frozen vertex $m+1$ in $Q_{u,\bm\beta}$ in correspondence to $d_L$.
 Therefore we have 
 $$
 H^{u,\bm\beta}_{m+1,m+1}=0,D^{u,\bm\beta}_{m+1,m+1}=-1,
 \widehat{B}^{u,\bm\beta}_{m+1,m+1}=-1.
 $$
Let $y',x'$ be two vertices in $Q_{u,\bm{\beta'}}$ corresponding to $y,x$ in $Q_{u,\bm\beta}$.
 We want to relate the matrices $\widehat{B^{u,\bm\beta}}=(\widehat{b}_{y,x})$ and $\widehat{B^{u,\bm\beta'}}=(\widehat{b'}_{y',x'})$ using row operations. 
 Lemma  \ref{lem: final lem} implies that for $x,y\neq m+1$
 \begin{equation}
 \label{eq: column operation}
\widehat{b'}_{y',x'}=\widehat{b}_{y,x}+\widehat{b}_{m+1,x}\widehat{b}_{y,m+1}=z_{y,x}.
 \end{equation}
 Thus the matrix $Z^{u,\bm\beta}$ is identical to $\widehat{B^{u,\bm{\beta'}}}$, except for the $(m+1)$-st row and $(m+1)$-st column and some reordering of rows and columns (see Remark \ref{rem: reordering}).
 Also, by construction $Z_{m+1,x}=0$ for $x\neq m+1$ and $Z_{m+1,m+1}=-1$.   
 $$
\widehat{B}^{u,\bm\beta}=\left(
\begin{matrix}
      & \vdots & & \vdots & \\
\cdots & -1 & \cdots & \widehat{b}_{m+1,x} & \cdots\\   
      & \vdots & & \vdots & \\
\cdots & \widehat{b}_{y,m+1} & \cdots & \widehat{b}_{y,x} & \cdots\\
     & \vdots & & \vdots & \\
\end{matrix}
 \right)\quad \text{and} \quad 
Z=\left(
\begin{matrix}
      & \vdots & & \vdots & \\
\cdots & -1 & \cdots & 0 & \cdots\\   
      & \vdots & & \vdots & \\
\cdots & \widehat{b}_{y,m+1} & \cdots & \widehat{b'}_{y',x'} & \cdots\\
     & \vdots & & \vdots & \\
\end{matrix}
 \right).
 $$
We conclude that 
 $$
 \det \widehat{B^{u,\bm\beta}}=\det Z^{u,\bm\beta}=-\det \widehat{B^{u,\bm{\beta'}}}=-(-1)^{m+f-1}=(-1)^{m+f}.
 $$
 For the second equality, we note that simultaneous reordering of rows and columns is equivalent to conjugation by a permutation matrix, and does not change determinant.  
 
 Finally, to prove that $\bhat$ has integer coefficients, we use the same inductive argument and observe that both $\widehat{b}_{m+1,x}$ and $\widehat{b}_{y,m+1}$ are always integers by Lemma \ref{lem: b m+1} below. Since $Z^{u,\bm\beta}$ has integer coefficients by assumption of induction, $\bhat$ has integer coefficients as well. 
\end{proof}

\begin{remark}\label{rmk: really full rank}
Theorem \ref{thm: main thm} implies that the integer span of columns of $\til$ is $\mathbb{Z}^m$,
which was previously proven in \cite{CGGLSS22,GLSBS22}.  In terminology of \cite{LS22}, the quiver $Q_{u,\bm\beta}$ has {\it really full rank}.
\end{remark}

We postpone proving key lemmas used in the proof of Theorem \ref{thm: main thm} to the subsequent subsections. Before we move on, we will see two examples that illustrate Theorem \ref{thm: main thm}.

\begin{exam}\label{ex: ex computation}
    Let us take $u=s_2$ and $\beta=\sigma_3\sigma_2\sigma_1\textcolor{blue}{\sigma_2}\sigma_3$ as an example. 
    \begin{figure*}[ht!]
    \centering
     \subfloat[]{
\begin{tikzpicture}[scale=1]
\draw [thick](-6.5,0) -- (-2,0);
\draw [thick](-6.5,1) -- (-2,1);
\draw [thick](-6.5,2) -- (-2,2);
\draw [thick](-6.5,3) -- (-2,3);

\draw [ultra thick] (-6,2.95) to (-6,2.05);
\node[scale =0.5, circle, draw=black, fill=white] (W) at (-6,2.95) {};
\node[scale =0.5, circle, draw=black, fill=black] (B) at (-6,2.05) {};
\node[color=black] (W) at (-6,3.2) {1};
\node[color=frenchrose] (1) at (-6.25,2.5) {1};
\node[color=frenchrose] (1) at (-6.1,1.5) {2,4};
\node[color=frenchrose] (1) at (-6.25,0.5) {3};

\draw [ultra thick] (-5.5,1.95) to (-5.5,1.05);
\node[scale =0.5, circle, draw=black, fill=white] (W) at (-5.5,1.95) {};
\node[scale =0.5, circle, draw=black, fill=black] (B) at (-5.5,1.05) {};
\node[color=black] (W) at (-5.5,2.2) {2};
\draw [ultra thick] (-5,0.95) to (-5,0.05);
\node[scale =0.5, circle, draw=black, fill=white] (W) at (-5,0.95) {};
\node[scale =0.5, circle, draw=black, fill=black] (B) at (-5,0.05) {};\
\node[color=black] (W) at (-5,1.2) {3};

\draw[color=white, line width=5] (-4,2) to (-3,2);
\draw[color=white, line width=5] (-4,1) to (-3,1);
\draw[color=blue, thick] (-4,2) to [out=0,in=180] (-3,1);
\draw[color=white, line width=5] (-4,1) to [out=0,in=180] (-3,2);
\draw[color=blue, thick] (-4,1) to [out=0,in=180] (-3,2);
\node[color=frenchrose] (1) at (-4.5,1.5) {4};

\draw [ultra thick] (-2.5,2.95) to (-2.5,2.05);
\node[scale =0.5, circle, draw=black, fill=white] (W) at (-2.5,2.95) {};
\node[scale =0.5, circle, draw=black, fill=black] (B) at (-2.5,2.05) {};
\node[color=frenchrose] (1) at (-4,2.5) {4};
,3.05);
\node[color=black] (W) at (-2.5,3.2) {4};
\end{tikzpicture}}
    \qquad
\subfloat[]{
\begin{tikzpicture}[squarednode/.style={rectangle, draw=blue!60, fill=blue!5, very thick, minimum size=5mm},
]
\node[draw, squarednode] (1) at (0,3){1};
\node[draw, squarednode] (2) at (0.5,2){2};
\node[draw, squarednode] (3) at (1,1){3};
\node[draw, squarednode] (4) at (2,3){4};
\draw[->,blue!75] (2) to (1);
\draw[->,blue!75] (1) to (4);
\draw[->,blue!75] (3) to (2);
\draw[->,blue!75] (4) to (3);
\end{tikzpicture}
}
\end{figure*}
    Note that all bridges are frozen bridges, and the \textcolor{red}{red} arrows in (B) indicates the half arrows between the frozen vertices.
    We consider soap films $C_1,C_2,C_3,C_4$ that correspond to a braid letter $\sigma_3,\sigma_2,\sigma_1$ and $\sigma_3$. We have $\p(C_1)=(0,0,1)$, $\p(C_2)=(0,1,0)$, $\p(C_3)=(1,0,0)$ and $\p(C_4)=(0,1,0)$. Then we can compute that $\p(C_1)\cdot\p(C_2)=\frac{1}{2}$, $\p(C_1)\cdot \p(C_3)=0$, $\p(C_1)\cdot \p(C_4)=\frac{1}{2}$. 
Continuing in this way, we obtain the boundary correction matrix 
    \[\D=\begin{bmatrix}
    -1 &  \frac{1}{2}  & 0 & \frac{1}{2}\\
    \frac{1}{2} & -1 & \frac{1}{2} & -1\\
    0 & \frac{1}{2} & -1 & \frac{1}{2}\\
    \frac{1}{2} & -1 & \frac{1}{2} & -1
\end{bmatrix}.
\]
Since half arrow matrix is equal to \[\Hmatrix=\begin{bmatrix}
    0 &  -\frac{1}{2}  & 0 & \frac{1}{2}\\
   \frac{1}{2} & 0 & -\frac{1}{2} & 0\\
    0 & \frac{1}{2} & 0 & -\frac{1}{2}\\
    -\frac{1}{2} & 0 & \frac{1}{2} & 0
\end{bmatrix},
\]
one has \[\bhat=\Hmatrix+\D=
\begin{bmatrix}
    -1 &  0  & 0 & 1\\
    1 & -1 & 0 & -1\\
    0 & 1 & -1 & 0\\
    0 & -1 & 1 & -1
\end{bmatrix}.
\]
Then, we see by the computation that the determinant of matrix 
is indeed $1$.
\end{exam}

\begin{exam}\label{ex: running ex computation}
    Let us revisit our running  Example \ref{ex: running ex intro}. One can read off that $\p(C_1)=\p(C_2)=\p(C_3)=\p(C_4)=(0,0,0,0,0)$,
    $\p(C_5)=(0,0,0,0,1)$, $\p(C_6)=(0,0,0,1,0)$, $\p(C_7)=(0,0,1,0,0)$, $\p(C_8)=(0,1,0,0,0)$, $\p (C_9)=(1,0,0,0,0)$ and $\p (C_{10})=(0,0,0,1,0)$. Then the boundary correction matrix is obtained by
    \[\D=\begin{bmatrix}
 0 \quad  0\\
    0 \quad N'
\end{bmatrix},
\]
where \[
N'=\begin{bmatrix}
 -1 &  \frac{1}{2}  & 0 & 0 & 0 & \frac{1}{2}\\
\frac{1}{2} & -1 & \frac{1}{2} & 0 & 0 & -1\\
 0 & \frac{1}{2} & -1 & \frac{1}{2} & 0 & \frac{1}{2}\\
0 &  0 & \frac{1}{2} & -1 & \frac{1}{2} & 0\\
0 & 0 & 0 & \frac{1}{2} & -1 & 0\\
 \frac{1}{2} & -1 & \frac{1}{2} & 0 & 0 & -1
\end{bmatrix}.
\]

Since the half arrow matrix is computed as \[\Hmatrix=\begin{bmatrix}
0 & -1 & 0 & 0 & 0 & -1 & 1 & 0 & 0 & 0\\
1 & 0 & -1 & 1 & 0 & 0 & -1 & 1 & 0 & 0\\
0 & 1 & 0 & -1 & 0 & 0 & 0 & -1 & 1 & 0\\
0 & -1 & 1 & 0 & 0 & 0 & 0 & 0 & 0 & 0\\
0 & 0 & 0 & 0 & 0 & -\frac{1}{2} & 0 & 0 & 0 & \frac{1}{2}\\
1 & 0 & 0 & 0 & \frac{1}{2} & 0 & -\frac{1}{2}& 0 & 0 & 0\\
-1 & 1 & 0 & 0 & 0 & \frac{1}{2} & 0 & -\frac{1}{2} & 0 & -\frac{1}{2}\\
0 & -1 & 1 & 0 & 0 & 0 & \frac{1}{2} & 0 & -\frac{1}{2} & 0\\
0 & 0 & -1 & 0 & 0 & 0 & 0 & \frac{1}{2} & 0 & 0\\
0 & 0 & 0 & 0 & -\frac{1}{2} & 0 & \frac{1}{2} & 0 & 0 & 0\\
\end{bmatrix},
\]
one has 
\[\bhat=\begin{bmatrix}
0 & -1 & 0 & 0 & 0 & -1 & 1 & 0 & 0 & 0\\
1 & 0 & -1 & 1 & 0 & 0 & -1 & 1 & 0 & 0\\
0 & 1 & 0 & -1 & 0 & 0 & 0 & -1 & 1 & 0\\
0 & -1 & 1 & 0 & 0 & 0 & 0 & 0 & 0 & 0\\
0 & 0 & 0 & 0 & -1 & 0 & 0 & 0 & 0 & 1\\
1 & 0 & 0 & 0 & 1 & -1 & 0 & 0 & 0 & -1\\
-1 & 1 & 0 & 0 & 0 & 1 & -1 & 0 & 0 & 0\\
0 & -1 & 1 & 0 & 0 & 0 & 1 & -1 & 0 & 0\\
0 & 0 & -1 & 0 & 0 & 0 & 0 & 1 & -1 & 0\\
0 & 0 & 0 & 0 & 0 & -1 & 1 & 0 & 0 & -1\\
\end{bmatrix},
\]
One can verify using Sage 
that $\det(\bhat)=1$.
\end{exam}

\begin{remark}
Note that Example \ref{ex: ex computation} and Example \ref{ex: running ex computation} are the cases when the number of frozen variables is unusually large, that is, the number of frozen bridges is bigger than the number of boundary regions in $G_{u,\bm\beta}$.
\end{remark}

\subsection{Auxiliary lemmas for Case $A$}\label{subsec: auxiliary lemmas A}

In this subsection, we assume the \textbf{Case A} situation in Theorem \ref{thm: main thm}. We assume that $1\not\in J_{u,\bm\beta}$. In other words, suppose that $\bm\beta=s_i\bm\beta'$ and $u=s_iu'$ where $u'$ is a subword of $\bm\beta'$. We introduce three auxiliary lemmas to prove Corollary \ref{cor: main cor for case a}, which was the key lemma when proving the \textbf{Case A} part of Theorem \ref{thm: main thm}.\\

We introduce a transformation $R_i:\Z^{n-1}\to \Z^{n-1}$ defined by 
$$
R_i(a_1,\ldots,a_n)=(a_1,\ldots,a_{i-1},-a_{i}+a_{i-1}+a_{i+1},a_{i+1},\ldots,a_{n-1})
$$
for $1\le i\le n-1$.

\begin{lem}\label{lem: reflection R_i}
For $a\in \Z^{n-1}$, we have $R_i(a)=a+2(a,e_{i})e_i$, where $e_i$ denotes the standard basis vector. Thus $R_i$ is a reflection in $e_i$ with respect to the symmetric bilinear form $(-,-)$.
\end{lem}
\begin{proof}
Assume that $a=(a_1,\ldots, a_{n-1})$. Then the direct computation gives us that
\[
(a,e_i)=-a_i+a_{i-1}+a_{i+1},
\]
and thus $a+2(a,e_i)e_i=R_i(a)$.
\end{proof}
\begin{lem}\label{lem: inner product under R}
For $a,b\in \mathbb{Z}^{n-1}$, we have
    $(R_i(a),R_i(b))=(a,b)$.
\end{lem}
\begin{proof}
From Lemma \ref{lem: reflection R_i}, we have 
\begin{align*}
(R_i(a),R_i(b)) & =(a+2(a,e_{i})e_i,b+2(b,e_{i})e_i)\\
       & =(a,b)+4(b,e_{i})(a,e_{i})+4(a,e_{i})(b,e_{i})(e_{i},e_{i})\\
       & =(a,b),\ \text{as}\ (e_{i},e_{i})=-1.
\end{align*}
\end{proof}
Now we need to understand how the boundary map $\partial_{u,\bm\beta}$ differ from $\partial_{u',\bm{\beta'}}$.

\begin{lem}\label{lem: case 1}
Let $C_{j}$ be the soap film of a vertex $j$ where $1\le j\le n$. Then 
$$
\partial_{u,\bm\beta}(C_{j})=R_i(\partial_{u',\bm{\beta'}}(C_{j})).
$$
\end{lem}
\begin{proof}
We consider possible cases for $\partial_{u,\bm\beta}(C_{j})$ shown in the Table \ref{tab:table for r_i}.
 Two cases when $$\partial_{u,\bm\beta}(C_{j})=(a_1,\ldots,a_{i-2},1,0,1,a_{i+2},\ldots, a_n)$$ and
 $$\partial_{u,\bm\beta}(C_{j})=(a_1,\ldots,a_{i-2},0,1,0,a_{i+2},\ldots, a_n)$$ do not happen due to the fact that we choose a positive distinguished subexpression for $u$ inside $\bm\beta$, see \cite{GLSBS22}. Note that in these cases one would get $a_i=2$ or $a_i=-1$ which contradicts the fact that $a_i\in \{0,1\}$. 
\end{proof}

\setlength{\arrayrulewidth}{0.5mm}
\setlength{\tabcolsep}{18pt}
\renewcommand{\arraystretch}{1.5}
\begin{table}[]
    \centering
\begin{tabular}{ |p{3cm}|p{3cm}|p{3cm}|}
\hline
\multicolumn{3}{|c|}{Triples $(a_{i-1},a_i,a_{i+1})$ in Lemma \ref{lem: case 1}} \\
\hline
$\partial_{u',\bm{\beta'}}(C_{j})$& $\partial_{u,\bm\beta}(C_{j})$ & $R_i(\partial_{u',\bm{\beta'}}(C_{j}))$ \\
\hline
(1,0,0) & (1,1,0)  & (1,1,0)  \\
(1,1,0) & (1,0,0)   & (1,0,0)  \\
(1,1,1)  & (1,1,1)  & (1,1,1)  \\
(0,0,0)  & (0,0,0)  & (0,0,0)  \\
(0,0,1)  & (0,1,1)  & (0,1,1)  \\
(0,1,1) & (0,0,1) & (0,0,1) \\
\hline
\end{tabular}
    \caption{Triples $(a_{i-1},a_i,a_{i+1})$ in Lemma 3.16}
    \label{tab:table for r_i}
\end{table}

As a combination of Lemma \ref{lem: inner product under R} with Lemma \ref{lem: case 1} above, we obtain the following Corollary \ref{cor: main cor for case a}.

\begin{cor}\label{cor: main cor for case a}
We have $D^{u,\bm\beta}=D^{u',\bm\beta'}$.
\end{cor}

\begin{proof}
We use the Definition \ref{defn: boundary} of the boundary correction matrix along with Lemma \ref{lem: inner product under R} and Lemma \ref{lem: case 1}, and obtain the following series of equalities: 
 \begin{align*}
D_{i,j}^{u,\beta} & =(\partial_{u,\bm\beta}(C_{i}),\partial_{u,\bm\beta}(C_{j}))\\
& = (R_{i}(\p_{u',\bm\beta'}(C_i)),R_{i}(\p_{u',\bm\beta'}(C_j)))\\
& = (\partial_{u',\bm{\beta'}}(C_{i}),\partial_{u',\bm{\beta'}}(C_{j}))\\
& = D^{u',\bm{\beta'}}_{i,j},
\end{align*}
 for all $1\le i,j\le m+f$. Therefore $D^{u,\bm\beta}=D^{u',\bm\beta'}$. 

\end{proof}


\subsection{Auxiliary lemmas for Case $B$}\label{subsec: auxiliary lemmas B}
In this subsection, we assume the \textbf{Case B} situation in Theorem \ref{thm: main thm}. We assume that $1\in J_{u,\bm\beta}$. To be specific, let us write $\bm\beta=s_i\bm\beta'$ and $u$ is a subword of $\bm\beta'$. We frequently make use of observations that was previously mentioned in Remark \ref{rmk: leftmost frozen} and Remark \ref{ex: leftmost frozen soap}: $1\in J_{u,\bm\beta}$ is in correspondence with $d_L$ and $C_{m+1}$. Note also that we have $\partial_{u,\bm\beta} (C_{m+1})=e_i$. Our goal is to prove Lemma \ref{lem: final lem} which was used in the proof of Theorem \ref{thm: main thm}. To achieve this goal, we introduce subsequent three lemmas.\\


Given a soap film $C_x$, we denote by $\mathfrak{a}_x,\mathfrak{b}_x,\mathfrak{c}_x,\mathfrak{d}_x$ its multiplicities in the regions near $d_L$, see Figure \ref{fig: H and F}. 
\begin{lem}
\label{lem: b m+1}
We have 
$$
H_{m+1,x}=-\frac{1}{2}(\mathfrak{a}_x+\mathfrak{c}_x)+\mathfrak{b}_x,\ {D}_{m+1,x}=(e_i,\partial (C_x))=-\mathfrak{d}_x+\frac{1}{2}\mathfrak{a}_x+\frac{1}{2}\mathfrak{c}_x, 
$$
and
$$
\widehat{b}_{m+1,x}=\mathfrak{b}_x-\mathfrak{d}_x,\ \widehat{b}_{x,m+1}=\mathfrak{a}_x+\mathfrak{c}_x-\mathfrak{b}_x-\mathfrak{d}_x.
$$
\end{lem}
\begin{proof}
The first equation follows from the definitions and Figure \ref{fig: H and F} using the fact that $C_{x}$ only covers the region in the left of bridge $m+1$. The second follows from 
$$
\widehat{b}_{m+1,x}=H_{m+1,x}+{D}_{m+1,x},\ \widehat{b}_{x,m+1}=-H_{m+1,x}+{D}_{m+1,x}
$$
which holds since $H$ is skew-symmetric and $D$ is symmetric.
\end{proof}

We study the difference between the soap films in 3D plabic graphs $G_{u,\bm\beta}$ and $G_{u,\bm{\beta'}}$, and the corresponding quivers $Q_{u,\bm\beta}$ and $Q_{u,\bm{\beta'}}$.
The soap films for mutable vertices of $Q_{u,\bm{\beta'}}$ do not change, but soap films for the frozen ones change according to the propagation rule across $d_L$. 
Let $x$ and $y$ be two vertices in $Q_{u,\bm\beta}$, and $x',y'$ be the corresponding 
vertices in $Q_{u,\bm{\beta'}}$.  
\begin{remark}
\label{rem: reordering}
Recall our indexing conventions in Remark \ref{rmk: ordering rmk}, where all mutable vertices are listed before all frozen ones. If $x'$ (resp. $y'$) is mutable then $x$ and $y$ are also mutable. If $x'$ (resp. $y'$) is frozen then $x$ (resp. $y$) could be either frozen or mutable. Overall, the relative order of the vertices which do not change type (mutable stay mutable or frozen stay frozen) is unchanged, but the frozen vertices which become mutable are inserted in the list of mutable ones, possibly in a complicated way. We also insert the new frozen vertex labeled by $m+1$.
\end{remark}
Let us first observe the difference between the half weight matrices $H_{y,x}:=\Hmatrix_{y,x}$ and $H'_{y',x'}:=H^{u,\bm\beta'}_{y',x'}$.

\begin{lem}\label{lem: compare H}
 Suppose $x,y\neq m+1$ and $x',y'$ correspond to $x,y$ as above.  We obtain \[
H'_{y',x'}=H_{y,x}+H_{m+1,x}{D}_{m+1,y}-{D}_{m+1,x}H_{m+1,y}.
\]
\end{lem}

\begin{proof}
  If $x'$ is mutable for $Q_{u,\bm{\beta'}}$ then $D_{m+1,x}=H_{m+1,x}=0$ and $H'_{y',x'}=H_{y,x}$, and the result holds. The case when $y'$ is mutable is similar, so from now on we assume that $x',y'$ are frozen in $Q_{u,\bm{\beta'}}$. 

In Figure \ref{fig: H and F}, we show the multiplicities of cycles corresponding to $x,y$ in various regions of $G_{u,\bm\beta}$. For the graph $G_{u,\bm\beta'}$ and a soap film $C_{x'}$, we simply erase $d_L$ and label the new region with $\mathfrak{d}_{x'}$. Note that the multiplicities $\mathfrak{a}_{x'},\mathfrak{b}_{x'},\mathfrak{c}_{x'}$ do not change.
    
    From Figure \ref{fig: H and F}, we observe that \[
    H_{y,x}=H'_{y',x'}-(\mathfrak{d}_x\mathfrak{b}_y-\mathfrak{d}_y\mathfrak{b}_x)-\frac{1}{2}(\mathfrak{d}_y-\mathfrak{b}_y)(\mathfrak{a}_x+\mathfrak{c}_x)+\frac{1}{2}(\mathfrak{d}_x-\mathfrak{b}_x)(\mathfrak{a}_y+\mathfrak{c}_y),
    \]
and by Lemma \ref{lem: b m+1} we get
\[
{D}_{m+1,y}=(e_i,\partial_{u,\bm\beta}( C_y))=-\mathfrak{d}_y+\frac{1}{2}\mathfrak{a}_y+\frac{1}{2}\mathfrak{c}_y,
\]
\[
{D}_{m+1,x}=(e_i,\partial_{u,\bm\beta}(C_x))=-\mathfrak{d}_x+\frac{1}{2}\mathfrak{a}_x+\frac{1}{2}\mathfrak{c}_x,
\]
\[
H_{m+1,x}=-\frac{1}{2}(\mathfrak{a}_x+\mathfrak{c}_x)+\mathfrak{b}_x,
\]
and \[
H_{m+1,y}=-\frac{1}{2}(\mathfrak{a}_y+\mathfrak{c}_y)+\mathfrak{b}_y.
\]
Then the direct computation enables us to verify that $H'_{y',x'}=H_{y,x}+H_{1,x}{D}_{1,y}-{D}_{1,x}H_{1,y}$.
\end{proof}


  


\begin{figure}
    \centering
 \begin{tikzpicture}
 \filldraw[color=piggypink] (1.5,0)  -- (4.5,0) -- (4.5,3) -- (1.5,3) -- cycle;
 \draw [thick] (1.5,0) to (4.5,0);
 \draw [thick] (1.5,1) to (4.5,1);
 \draw [thick] (1.5,2) to (4.5,2);
 \draw [thick] (1.5,3) to (4.5,3); 
\draw [ultra thick](3,1) to (3,2);
 \node [scale =0.5, circle, draw=black, fill=black] (b) at (3,1) {};
\node [scale =0.5, circle, draw=black, fill=white] (w) at (3,2) {};
\node (a) at (3,2.5) {$\mathfrak{a}_x, \mathfrak{a}_y$};
 \node (a) at (4,1.5) {$\mathfrak{b}_x, \mathfrak{b}_y$};
  \node (a) at (2,1.5) {$\mathfrak{d}_x, \mathfrak{d}_y$};
   \node (a) at (3,0.5) {$\mathfrak{c}_x, \mathfrak{c}_y$};
 \draw[color=red, {Stealth[red]}-,thick] (3.5,1.7)-- (2.5,1.7)  ;
\draw[color=red, {Stealth[red]}-,thick] (3.5,1.3) -- (2.5,1.3) ;
\draw [color=red, {Stealth[red]}-,thick] (3.1, 2.3)-- (3.75, 1.8) ;
\draw [color=red, {Stealth[red]}-,thick] (2.25, 1.8)-- (2.9, 2.3) ;

\draw [color=red, {Stealth[red]}-,thick] (2.25, 1.2)--(2.9, 0.7);
\draw [color=red, {Stealth[red]}-,thick] (3.1,0.7)--(3.75,1.2);
 \end{tikzpicture}
 \caption{Local picture related to $x$ and $y$}
    \label{fig: H and F}
 \end{figure}

Now we compare $\p_{u,\bm\beta} (C_x)$ with $\p_{u,\bm\beta'} (C_{x'})$. 
Note that $C_x$ and $C_{x'}$ only differ at $d_L$; we remove $d_L$ and obtain $G_{u,\bm{\beta'}}$.

\begin{lem}\label{lem: compare partial}
We have \begin{align*}
    \p_{u,\bm\beta'}(C_{x'})=\p_{u,\bm\beta} (C_x)+\widehat{b}_{m+1,x}e_i,\\
     \p_{u,\bm\beta'}(C_{y'})=\p_{u,\bm\beta} (C_y)+\widehat{b}_{m+1,y}e_i.
\end{align*}
\end{lem}

\begin{proof}
One observes that $\p_{u,\bm \beta'} (C_{x'})=\p_{u,\bm\beta} (C_x)+(\mathfrak{b}_x-\mathfrak{d}_x)e_i$. Then by Lemma \ref{lem: b m+1} we have $\widehat{b}_{m+1,x}=\mathfrak{b}_x-\mathfrak{d}_x$, and therefore we obtain $\p_{u,\bm\beta'}(C_{x'})=\p_{u,\bm\beta} (C_x)+\widehat{b}_{m+1,x}e_i$. Similarly, we have $\p_{u,\bm\beta'}(C_{y'})=\p_{u,\bm\beta} (C_y)+\widehat{b}_{m+1,y}e_i$.
\end{proof}

\begin{remark}
    Alternatively, one can observe that Lemma \ref{lem: compare partial} holds from Table \ref{tab:table for compare partial} using the propagation rule in Figure \ref{fig: propagation rule}.
\end{remark}

\setlength{\arrayrulewidth}{0.5mm}
\setlength{\tabcolsep}{18pt}
\renewcommand{\arraystretch}{1.5}
\begin{table}[]
    \centering
\begin{tabular}{ |p{2cm}|p{2cm}|p{2cm}|p{2cm}|p{2cm}|}
\hline
\multicolumn{5}{|c|}{Triples $(a_{i-1},a_i,a_{i+1})$ in Lemma \ref{lem: compare partial}} \\
\hline
$\partial_{u',\bm{\beta'}}(C_{x'})$ & $\partial_{u,\bm\beta}(C_{x})$ & $D_{m+1,x}e_i$ & $H_{m+1,x}e_i$ & $\hat{b}_{m+1,x}e_i$ \\
\hline
(0,1,0) & (0,0,0)  & (0,0,0) & (0,1,0) & (0,1,0)  \\
(1,1,0) & (1,0,0)  & (0,$\frac{1}{2}$,0) & (0,$\frac{1}{2}$,0) & (0,1,0)  \\
(0,1,1) & (0,0,1)  & (0,$\frac{1}{2}$,0) & (0,$\frac{1}{2}$,0) & (0,1,0)  \\
(1,1,1) & (1,0,1)  & (0,0,0) & (0,1,0) & (0,1,0)  \\
(1,1,0) & (1,1,0)  & (0,-$\frac{1}{2}$,0) & (0,$\frac{1}{2}$,0) & (0,0,0)  \\
(1,0,0) & (1,0,0)  & (0,-$\frac{1}{2}$,0) & (0,$\frac{1}{2}$,0) & (0,0,0)  \\
(1,1,0) & (1,0,0)  & (0,0,0) & (0,1,0) & (0,0,0)  \\
\hline
\end{tabular}
    \caption{Triples $(a_{i-1},a_i,a_{i+1})$ in Lemma \ref{lem: compare partial}. For example, the computations in the top row is from the first figure of Figure \ref{fig: propagation rule}.}
    \label{tab:table for compare partial}
\end{table}

Now we are able to compare $\bhat$ with $\widehat{B^{u,\bm{\beta'}}}$. 
\begin{lem}\label{lem: final lem}
We have \[
\widehat{b'}_{y',x'}=\widehat{b}_{y,x}+\widehat{b}_{m+1,x}\widehat{b}_{y,m+1}.
\]
\end{lem}

\begin{proof}
By Definition \ref{defn: boundary} and Lemma \ref{lem: compare partial}, we have
\begin{align*}
D'_{y',x'} & =(\p_{u,\bm\beta'}(C'_{y'}),\p_{u,\bm\beta'} (C'_{x'}))\\
& = \left(\p_{u,\bm\beta} (C_y)+\widehat{b}_{m+1,y}e_i, \p _{u,\bm\beta}(C_x)+\widehat{b}_{m+1,x}e_i\right)\\
& = (\p_{u,\bm\beta} (C_y),\p_{u,\bm\beta} (C_{x}))+\widehat{b}_{m+1,y}(e_i,\p_{u,\bm\beta} (C_x))+\widehat{b}_{m+1,x}(\p_{u,\bm\beta} (C_y),e_i)+\widehat{b}_{m+1,y}\widehat{b}_{m+1,x}(e_i,e_i)\\
& =D_{y,x}+\widehat{b}_{m+1,y}D_{m+1,x}+\widehat{b}_{m+1,x}D_{m+1,y}-\widehat{b}_{m+1,y}\widehat{b}_{m+1,x}\\
& = D_{y,x}+D_{m+1,x}D_{m+1,y}-H_{m+1,x}H_{m+1,y}.
\end{align*}

Therefore, by Lemma \ref{lem: compare H} we have \begin{align*}
\widehat{b'}_{y',x'}& =H'_{y',x'}+D'_{y',x'}\\
& = (H_{y,x}+H_{m+1,x}D_{m+1,y}-D_{m+1,x}H_{m+1,y})+ (D_{y,x}+D_{m+1,x}D_{m+1,y}-H_{m+1,x}H_{m+1,y})\\
& = \widehat{b}_{y,x}+\widehat{b}_{m+1,x}(D_{m+1,y}-H_{m+1,y})\\
& = \widehat{b}_{y,x}+\widehat{b}_{m+1,x}\widehat{b}_{y,m+1},
\end{align*}
as desired. The last equality follows from Remark \ref{rem: D is sym} and the fact that $\Hmatrix$ is a skew symmetric matrix.
\end{proof}

\section{Application: Cluster automorphism group of braid varieties}

\subsection{Cluster automorphism group of braid varieties}
We recall the definition and important properties of cluster automorphism group of braid varieties, discussed in \cite{clusterloci, LS22}.
\begin{defn}\cite[Section 5.1]{LS22}\label{def: cluster auto}
Let $\mathcal{A}$ be a cluster algebra.
A cluster automorphism of $\mathcal{A}$ is an algebra automorphism $\varphi: \mathcal{A}\to \mathcal{A}$ such that for every cluster variable $z\in \mathcal{A}$, $\varphi(z)$ is a non-zero scalar multiple of $z$. The collection of cluster automorphisms of $\mathcal{A}$ is denoted by $\text{Aut}(\mathcal{A})$.
\end{defn}

Recall that Theorem \ref{thm: main thm} implies that $\det(\bhat)=(-1)^{m+f}$, and we use $\ahat$ to indicate its inverse matrix. Recall also from the introduction that we used ${col}_j(\ahat)$ to denote the $j$th column vector of $\ahat$.

\begin{lem}\label{lem: b tilde and ahat}
The $\ker(\til)$ in $\Z^{m+f}$ is a free abelian group of rank $f$ with the basis given by the column vectors
$\text{col}_{j}(\ahat)$ for $m+1\le j\le m+f$.
\end{lem}

\begin{proof}
Consider a vector $v\in \Z^{m+f}$, then $\til v$ is given by the first $m$ coordinates of $\bhat v$. Since ${col}_j(\ahat)$
form a basis of $\Z^{m+f}$, we can expand $v$ in terms of this basis, and write $v=\ahat w$ for some $w\in \Z^{m+f}$. Then
$\bhat v=\bhat \ahat w=w$. 
We conclude that $v$ is in $\ker(\til)$ if and only if the first $m$ coordinates of $w$ vanish, and the result follows.

\end{proof}
Given an $k\times l$ integer matrix $M$, we consider a map $\text{mult}: (\mathbb{C}^\ast)^l\to (\mathbb{C})^k$ such that $$
\text{mult}(M)(t_1,\ldots,t_l)=\Pi_{j=1}^{l}t_j^{M_{i,j}}.
$$
Lam and Speyer \cite{LS22} related the properties of the matrix $\widetilde{B}$ to the action of an algebraic torus on the cluster variety in the following way.

\begin{thm}\cite[Proposition 5.1]{LS22}\label{thm: lam-speyer}
Given any seed $(x,\tilde{B})$, the cluster automorphism group and its action on cluster varieties is described via $\tilde{B}$ matrix in the following way.

\begin{enumerate}
    \item There is an automorphism $\varphi\in \text{Aut}(\mathcal{A})$ such that $\phi(x_i)=t_ix_i$ if and only if $v=(t_1,\ldots, t_{m+f})\in \text{ker}(\text{mult}(\tilde{B}))$.
    \item 
    Suppose that $v\in \text{ker}(\text{mult}(\tilde{B}))$. Then there is a $\CC^\times$ action on the braid variety $\X$ given by $x_i\mapsto t^{a_i}x_i$ where $1\le i\le m+f$. This action is a quasi-cluster homomorphism in the sense of \cite{Fraser}.
\end{enumerate}
\end{thm}

\begin{remark}
We rephrased the statement in \cite[Proposition 5.1]{LS22} based on Remark \ref{rmk: really full rank}.
\end{remark}

Combining Lemma \ref{lem: b tilde and ahat} with Theorem \ref{thm: lam-speyer}, we describe the cluster automorphism group and its action on braid varieties as follows.

\begin{cor}\label{cor: torus}
Let us fix $Q_{u,\bm\beta}$ as before.
For $1\le j\le f$,
   let us write ${col}_{m+j}(\ahat)=(a_{1,m+j},\ldots,a_{m+f,m+j})$.
    Then $\mathrm{Aut}(\mathcal{A}(Q_{u,\bm\beta}))$ is an algebraic torus with coordinates $t_1,\ldots,t_f$ which acts on $\X$ by 
    \[
x_i\to \prod t_j^{a_{i,m+j}}\cdot x_i,
\]
where $\textbf{x}=(x_1,\ldots, x_{m+f})$ is an initial seed of $\X$.
\end{cor}

\begin{remark}
When $u$ is an identity permutation, then the number of frozen variables are equal to the number of distinct Artin generators appearing in $\beta$. 
However, in general, the number of frozen vertices is bigger or equal to the number of distinct Artin generators appearing in $\beta$. See Example \ref{ex: running ex computation} and Example \ref{ex: double arrow} in
Section \ref{sec: ex of last sec}.
\end{remark}
Lastly, we give an \textbf{inductive description} of the matrix $\ahat$. 
 Suppose that $\bm\beta=\sigma_i\bm\beta'$ and $u$ is a subword of $\bm\beta'$, as in {\bf Case B} as in Section \ref{subsec: auxiliary lemmas B}. Recall the definition of $\widehat{Z}^{u,\bm\beta}$ matrix from Definition \ref{defn: Z hat}.

\begin{lem}\label{lem: inductive factorization for inverse}
We have $\ahat=R(\widehat{Z}^{u,\bm\beta})^{-1}L$ where the matrices $R$ and $L$ are as in Lemma \ref{lem: L and R}.
\end{lem}

\begin{proof}
Note that  $(\widehat{Z}^{u,\beta})^{-1}$ is obtained from ${A}^{u,\bm{\beta'}}$ by adding $(m+1)$-st row and column with $(-1)$ on diagonal and $0$ elsewhere.
By Lemma \ref{lem: L and R} we have $L\bhat R=\widehat{Z}^{u,\beta}$, so $R^{-1}\ahat L^{-1}=(\widehat{Z}^{u,\beta})^{-1}$ and $\ahat=R(\widehat{Z}^{u,\beta})^{-1}L$.
\end{proof}






\subsection{Examples}\label{sec: ex of last sec}

We explicitly describe the action of $\text{Aut}(A(Q_{u,\bm\beta}))$ on $X_{u,\bm\beta}$ with several examples, including the running example from Example \ref{ex: running ex computation}. Observing from Example \ref{ex: inverse from 3.8} and Example \ref{ex: double arrow}, one might wonder whether \textit{all} entries in $\ahat$ matrix have the same sign. It seems that this \newword{sign phenomenon} is true for all $X_{id,\beta}$, but not in general, see Example \ref{ex: counter to sign}.

Suppose that we are given with a $(m+f)$-tuple of initial cluster variables $\textbf{x}=(x_i)_{1\le i\le m+f}$ where the labels comes from Remark \ref{rmk: ordering rmk} and $f$ counts the number of frozen variables in $A(Q_{u,\beta})$. 
Throughout this section, black (resp. blue) colored arrows are normal (resp. half) arrows, and double arrow is indicated with $2$. We continue using the coloring convention on vertices as in Figure \ref{fig:running ex quiver}.


\begin{exam}\label{ex: inverse from 3.8}
Continuing with our running Example \ref{ex: running ex computation}, we can compute $\ahat$ as below. From this, Corollary \ref{cor: torus} implies that there is a $(\CC^\times)^6$ action on $\X$ such that
\begin{align*}
x_1& \to t_3^{-1}x_1,  &    x_2&\to t_3^{-1}t_4^{-1}t_{5}^{-1}x_2,  & x_3&\to t_3^{-1}t_4^{-1}t_{5}^{-1}x_3,     &  x_{4}&\to t_5^{-1}x_{4},\\
x_5&\to t_1^{-1}t_3^{-1}t_4^{-1}t_5^{-1}t_{6}^{-1}x_5, &        x_6&\to t_1^{-1}t_2^{-1}t_3^{-1}x_6,   &  x_7&\to t_1^{-1}t_2^{-1}t_3^{-2}t_4^{-1}t_5^{-1}x_7,\\
x_8&\to t_1^{-1}t_2^{-1}t_3^{-2}t_4^{-2}t_5^{-1}x_8,   &  x_9&\to t_1^{-1}t_2^{-1}t_3^{-1}t_4^{-1}t_5^{-1}x_9,          &  x_{10}&\to t_3^{-1}t_4^{-1}t_5^{-1}t_{6}^{-1}x_{10}.
\end{align*}
\setcounter{MaxMatrixCols}{15}

\[\ahat=
\left[
\begin{array}{cccc|cccccc}
-1 & 0 & 0 & 0 & 0 & 0 & -1 & 0 & 0 & 0\\
-1 & -1 & -1 & -1 & 0 & 0 & -1 & -1 & -1 & 0\\
-1 & -1 & -1 & 0 & 0 & 0 & -1 & -1 & -1 & 0\\
0 & 0 & -1 & -1 & 0 & 0 & 0 & 0 & -1 & 0\\
0 & -1 & -1 & -1 & -1 & 0 & -1 & -1 & -1 & -1\\
-1 & 0 & 0 & 0 & -1 & -1 & -1 & 0 & 0 & 0\\
-1 & -1 & -1 & -1 & -1 & -1 & -2 & -1 & -1 & 0\\
-1 & -1 & -1 & 0 & -1 & -1 & -2 & -2 & -1 & 0\\
0 & 0 & 0 & 0 & -1 & -1 & -1 & -1 & -1 & 0\\
0 & -1 & -1 & -1 & 0 & 0 & -1 & -1 & -1 & -1\\
\end{array}
\right]
\]

\end{exam}

\begin{exam}\label{ex: double arrow}
Consider $u=s_3s_1s_2s_5s_4$ and $\bm\beta=(1,1,5,5,3,3,2,2,4,3,2,1,2,5,4,3)$. The right figure below is a picture of $Q_{u,\bm\beta}$ which has a double arrow from a vertex $6$ to a vertex $10$. Note that $Q_{u,\bm\beta}$ is not a plabic quiver considered in \cite{GL24}. 
\begin{figure*}[ht!]
    \centering
\begin{tikzpicture}[scale=1]
\draw [thick](-0.5,0) -- (8,0);
\draw [thick](-0.5,1) -- (8,1);
\draw [thick](-0.5,2) -- (8,2);
\draw [thick](-0.5,3) -- (8,3);
\draw [thick](-0.5,4) -- (8,4);
\draw [thick](-0.5,5) -- (8,5);

\draw [ultra thick] (0,0.95) to (0,0.05);
\node[scale =0.5, circle, draw=black, fill=white] (W) at (0,0.95) {};
\node[scale =0.5, circle, draw=black, fill=black] (B) at (0,0.05) {};\node[color=black] (W) at (0,1.2) {7};

\node[color=frenchrose] (1) at (-0.25, 0.5) {7};

\draw [ultra thick] (0.5,0.95) to (0.5,0.05);
\node[scale =0.5, circle, draw=black, fill=white] (W) at (0.5,0.95) {};
\node[scale =0.5, circle, draw=black, fill=black] (B) at (0.5,0.05) {};
\node[color=frenchrose] (1) at (0.25,0.5) {1};
\node[color=black] (W) at (0.5,1.2) {1};

\draw [ultra thick] (1,4.95) to (1,4.05);
\node[scale =0.5, circle, draw=black, fill=white] (W) at (1,4.95) {};
\node[scale =0.5, circle, draw=black, fill=black] (B) at (1,4.05) {};
\node[color=frenchrose] (1) at (-0.25,4.5) {8};
\node[color=black] (W) at (1,5.2) {8};
\draw [ultra thick] (1.5,4.95) to (1.5,4.05);
\node[scale =0.5, circle, draw=black, fill=white] (W) at (1.5,4.95) {};
\node[scale =0.5, circle, draw=black, fill=black] (B) at (1.5,4.05) {};
\node[color=frenchrose] (1) at (1.25,4.5) {2};
\node[color=black] (W) at (1.5,5.2) {2};


\draw [ultra thick] (2,2.95) to (2,2.05);
\node[scale =0.5, circle, draw=black, fill=white] (W) at (2,2.95) {};
\node[scale =0.5, circle, draw=black, fill=black] (B) at (2,2.05) {};
\node[color=frenchrose] (1) at (0.5,2.5) {9};
\node[color=black] (W) at (2,3.2) {9};

\draw [ultra thick] (2.5,2.95) to (2.5,2.05);
\node[scale =0.5, circle, draw=black, fill=white] (W) at (2.5,2.95) {};
\node[scale =0.5, circle, draw=black, fill=black] (B) at (2.5,2.05) {};
\node[color=frenchrose] (1) at (2.25,2.5) {3};
\node[color=black] (W) at (2.5,3.2) {3};

\draw [ultra thick] (3,1.95) to (3,1.05);
\node[scale =0.5, circle, draw=black, fill=white] (W) at (3,1.95) {};
\node[scale =0.5, circle, draw=black, fill=black] (B) at (3,1.05) {};

\node[color=frenchrose] (1) at (1,1.5) {10};
\node[color=black] (W) at (3,2.2) {10};

\draw [ultra thick] (3.5,1.95) to (3.5,1.05);
\node[scale =0.5, circle, draw=black, fill=white] (W) at (3.5,1.95) {};
\node[scale =0.5, circle, draw=black, fill=black] (B) at (3.5,1.05) {};
\node[color=frenchrose] (1) at (3.25,1.5) {4};
\node[color=black] (W) at (3.7,2.2) {4};

\draw [ultra thick] (4,3.95) to (4,3.05);
\node[scale =0.5, circle, draw=black, fill=white] (W) at (4,3.95) {};
\node[scale =0.5, circle, draw=black, fill=black] (B) at (4,3.05) {};
\node[color=frenchrose] (1) at (1.5,3.5) {11};
\node[color=black] (W) at (3.7,3.8) {11};

 \draw[color=white, line width=5] (4.2,2) to (4.8,2);
 \draw[color=white, line width=5] (4.2,3) to (4.8,3);
\draw[color=blue, thick] (4.2,3) to [out=0,in=180] (4.8,2);
\draw[color=white, line width=5] (4.2,2) to [out=0,in=180] (4.8,3);
\draw[color=blue, thick] (4.2,2) to [out=0,in=180] (4.8,3);

\draw [ultra thick] (5,1.95) to (5,1.05);
\node[scale =0.5, circle, draw=black, fill=white] (W) at (5,1.95) {};
\node[scale =0.5, circle, draw=black, fill=black] (B) at (5,1.05) {};
\node[color=frenchrose] (1) at (4.25,1.5) {5,6};
\node[color=frenchrose] (1) at (3.5,2.5) {5,6};
\node[color=black] (W) at (5,2.2) {5};

 \draw[color=white, line width=5] (5.2,0) to (5.8,0);
 \draw[color=white, line width=5] (5.2,1) to (5.8,1);
\draw[color=blue, thick] (5.2,1) to [out=0,in=180] (5.8,0);
\draw[color=white, line width=5] (5.2,0) to [out=0,in=180] (5.8,1);
\draw[color=blue, thick] (5.2,0) to [out=0,in=180] (5.8,1);

 \draw[color=white, line width=5] (5.8,2) to (6.4,2);
 \draw[color=white, line width=5] (5.8,1) to (6.4,1);
\draw[color=blue, thick] (5.8,2) to [out=0,in=180] (6.4,1);
\draw[color=white, line width=5] (5.8,1) to [out=0,in=180] (6.4,2);
\draw[color=blue, thick] (5.8,1) to [out=0,in=180] (6.4,2);

 \draw[color=white, line width=5] (6.2,5) to (6.8,5);
 \draw[color=white, line width=5] (6.2,4) to (6.8,4);
\draw[color=blue, thick] (6.2,5) to [out=0,in=180] (6.8,4);
\draw[color=white, line width=5] (6.2,4) to [out=0,in=180] (6.8,5);
\draw[color=blue, thick] (6.2,4) to [out=0,in=180] (6.8,5);

 \draw[color=white, line width=5] (6.8,4) to (7.4,4);
 \draw[color=white, line width=5] (6.8,3) to (7.4,3);
\draw[color=blue, thick] (6.8,4) to [out=0,in=180] (7.4,3);
\draw[color=white, line width=5] (6.8,3) to [out=0,in=180] (7.4,4);
\draw[color=blue, thick] (6.8,3) to [out=0,in=180] (7.4,4);

\draw [ultra thick] (7.5,2.95) to (7.5,2.05);
\node[scale =0.5, circle, draw=black, fill=white] (W) at (7.5,2.95) {};
\node[scale =0.5, circle, draw=black, fill=black] (B) at (7.5,2.05) {};
\node[color=black] (W) at (7.5,3.2) {6};
\node[color=frenchrose] (1) at (5.5,2.5) {6};
\node[color=frenchrose] (1) at (3,0.5) {6};
\node[color=frenchrose] (1) at (5.5,3.5) {6};
\node[color=frenchrose] (1) at (4,4.5) {6};

\end{tikzpicture}
    \qquad
\begin{tikzpicture}[squarednode/.style={rectangle, draw=blue!60, fill=blue!5, very thick, minimum size=5mm},
]

\node[draw, circle,red] (16) at (4,3){6};
\node[draw, circle, red] (11) at (1,3){5};
\node[draw, circle,red] (8) at (0,1){4};
\node[draw, squarednode] (7) at (1,1){10};
\node[draw, squarednode] (9) at (2,1){11};
\node[draw, circle, red] (6) at (3,1){3};
\node[draw, circle,red] (4) at (4,1){2};
\node[draw, circle,red] (2) at (5,1){1};
\node[draw, squarednode] (5) at (3,0){9};
\node[draw, squarednode] (3) at (4,0){8};
\node[draw, squarednode] (1) at (4,-1.5){7};

 \draw (16) edge[->] node[font=\tiny\ttfamily,above] {2} (7);

\draw[->] (8) to (16);
\draw[->] (6) to (16);
\draw[->] (4) to (16);
\draw[->] (2) to (16);
\draw[->] (16) to (9);
\draw[->] (8) to (11);
\draw[->] (6) to (11);
\draw[->] (11) to (7);
\draw[->] (11) to (9);
\draw[->] (7) to (8);  
\draw[->] (5) to (6);
\draw[->] (3) to (4);
\draw[->] (1) to (2);
\draw[->,blue!75] (9) to (3);  
\draw[->,blue!75] (9) to (5);
\draw[->,blue!75] (7) to (5);
\draw[->,blue!75] (7) to (1);  
\end{tikzpicture}
    \label{fig:ex7}
\end{figure*}
Observe that all entries in $\ahat$ are either all negative integers or $0$. Corollary \ref{cor: torus} implies that there is a $(\CC^\times)^5$ on $\X$ such that

\begin{align*}
x_1 & \to t_1^{-1}t_2^{-1}t_3^{-1}t_3^{-1}t_5^{-1}x_1,  &  x_2 &\to t_3^{-1}t_4^{-1}t_5^{-1}x_2,  & x_3&\to t_1^{-1}t_2^{-1}t_3^{-2}t_4^{-1}t_5^{-2}x_3, \\
  x_{4}&\to t_1^{-1}t_2^{-1}t_3^{-1}t_4^{-2}t_5^{-2}x_4,
& x_{5}&\to t_1^{-1}t_3^{-1}t_4^{-1}t_5^{-1}x_{5},
& x_{6}&\to t_2^{-1}t_3^{-1}t_4^{-1}t_5^{-1}x_{6}\\
  x_7 &\to t_1^{-1}t_2^{-1}t_3^{-1}t_4^{-1}t_5^{-1}x_7,
&  x_8 &\to t_1^{-1}t_2^{-1}t_3^{-1}t_4^{-1}t_5^{-2}x_8,
&  x_9 &\to t_1^{-1}t_2^{-1}t_3^{-2}t_4^{-2}t_5^{-2}x_9,\\
  x_{10} &\to t_1^{-1}t_2^{-1}t_3^{-2}t_4^{-2}t_5^{-2}x_{10},
&  x_{11} &\to t_1^{-1}t_2^{-1}t_3^{-1}t_4^{-1}t_5^{-1}x_{11}.
\end{align*}
\setcounter{MaxMatrixCols}{15}

Here, the computation of a matrix ${\ahat}$ is given as below.


\[
\left[
\begin{array}{cccccc|ccccc}
0 & 0 & -1 & -1 & 0 & -1     & -1 & -1 & -1 & -1 & -1\\
-1  & -1 & -1 & 0 & -1 & 0   & 0  & 0  & -1 & -1 & -1 \\
-0 & -1 & -1 & -2 & -1 &  -1 & -1 & -1 & -2 & -1 & -2\\
-1 & -1 & -3 & -1 & -1 & -1  & -1 & -1 & -1 & -2 & -2\\
0 &  -1 & -1 & -1 & 0 &  -1  & -1 & 0  & -1 & -1 & -1\\
-1 &  0 & -1 & 0 & -1 & -1   & 0  & -1 & -1 & -1 &  -1\\
0 &  -1 & -1 & -1 & 0 &  -1  & -1 & -1 & -1 & -1 & -1\\
0 & -1 & -2 & -1 & 0  & -1   & -1 & -1 & -1 & -1 & -2\\
-1 & -1 & -3 & -1 & -1 & -1  & -1 & -1 & -2 & -2 & -2\\
-1 & -1 & -2 & -2 & -1 & -1  & -1 & -1 & -2 & -2 & -2\\
-1 & 0 & -1 & -1 & 0 & -1    & -1 & -1 & -1 & -1 & -1
\end{array}
\right].
\]

\end{exam}

Notice that the sign phenomenon fails in the following example. 

\begin{exam}\label{ex: counter to sign}
    Let us consider $u=s_1s_2$ and $\bm\beta=(1,3,1,2,1,3,2,2,3)$.
     Then the figure of $G_{u,\bm\beta}$ and $Q_{u,\bm\beta}$ (allowing half arrows between frozen vertices) are drawn as below. 
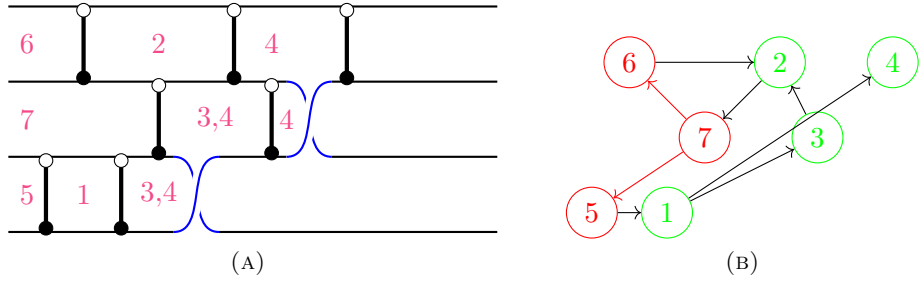
    \begin{figure*}[h]
    \centering
     \subfloat[]{
\begin{tikzpicture}[scale=1]
\draw [thick](-6.5,0) -- (0,0);
\draw [thick](-6.5,1) -- (0,1);
\draw [thick](-6.5,2) -- (0,2);
\draw [thick](-6.5,3) -- (0,3);

\draw [ultra thick] (-6,0.95) to (-6,0.05);
\node[scale =0.5, circle, draw=black, fill=white] (W) at (-6,0.95) {};
\node[scale =0.5, circle, draw=black, fill=black] (B) at (-6,0.05) {};
\node[color=black] (W) at (-6,1.2) {5};
\node[color=frenchrose] (1) at (-6.25,0.5) {5};

\draw [ultra thick] (-5.5,2.95) to (-5.5,2.05);
\node[scale =0.5, circle, draw=black, fill=white] (W) at (-5.5,2.95) {};
\node[scale =0.5, circle, draw=black, fill=black] (B) at (-5.5,2.05) {};
\node[color=frenchrose] (1) at (-6.25,2.5) {6};
\node[color=black] (W) at (-5.5,3.2) {6};

\draw [ultra thick] (-5,0.95) to (-5,0.05);
\node[scale =0.5, circle, draw=black, fill=white] (W) at (-5,0.95) {};
\node[scale =0.5, circle, draw=black, fill=black] (B) at (-5,0.05) {};
\node[color=frenchrose] (1) at (-5.5,0.5) {1};
\node[color=black] (W) at (-5,1.2) {1};

\draw [ultra thick] (-4.5,1.95) to (-4.5,1.05);
\node[scale =0.5, circle, draw=black, fill=white] (W) at (-4.5,1.95) {};
\node[scale =0.5, circle, draw=black, fill=black] (B) at (-4.5,1.05) {};
\node[color=frenchrose] (1) at (-6.25,1.5) {7};
\node[color=black] (W) at (-4.7,1.8) {7};


 \draw[color=white, line width=5] (-4.3,0) to (-3.7,0);
 \draw[color=white, line width=5] (-4.3,1) to (-3.7,1);
\draw[color=blue, thick] (-4.3,1) to [out=0,in=180] (-3.7,0);
\draw[color=white, line width=5] (-4.3,0) to [out=0,in=180] (-3.7,1);
\draw[color=blue, thick] (-4.3,0) to [out=0,in=180] (-3.7,1);

\draw [ultra thick] (-3.5,2.95) to (-3.5,2.05);
\node[scale =0.5, circle, draw=black, fill=white] (W) at (-3.5,2.95) {};
\node[scale =0.5, circle, draw=black, fill=black] (B) at (-3.5,2.05) {};
\node[color=frenchrose] (1) at (-4.5,2.5) {2};
\node[color=black] (W) at (-3.5,3.2) {2};
\draw [ultra thick] (-3,1.95) to (-3,1.05);
\node[scale =0.5, circle, draw=black, fill=white] (W) at (-3,1.95) {};
\node[scale =0.5, circle, draw=black, fill=black] (B) at (-3,1.05) {};
\node[color=frenchrose] (1) at (-3.75,1.5) {3,4};
\node[color=frenchrose] (1) at (-4.5,0.5) {3,4};
\node[color=black] (W) at (-2.8,2.2) {3};


 \draw[color=white, line width=5] (-2.8,2) to (-2.2,2);
 \draw[color=white, line width=5] (-2.8,1) to (-2.2,1);
\draw[color=blue, thick] (-2.8,2) to [out=0,in=180] (-2.2,1);
\draw[color=white, line width=5] (-2.8,1) to [out=0,in=180] (-2.2,2);
\draw[color=blue, thick] (-2.8,1) to [out=0,in=180] (-2.2,2);

\draw [ultra thick] (-2,2.95) to (-2,2.05);
\node[scale =0.5, circle, draw=black, fill=white] (W) at (-2,2.95) {};
\node[scale =0.5, circle, draw=black, fill=black] (B) at (-2,2.05) {};
\node[color=frenchrose] (1) at (-3,2.5) {4};
\node[color=frenchrose] (1) at (-2.8,1.5) {4};
\node[color=black] (W) at (-2,3.2) {4};

\end{tikzpicture}}
    \qquad
\subfloat[]{
\begin{tikzpicture}[squarednode/.style={rectangle, draw=blue!60, fill=blue!5, very thick, minimum size=5mm},
]
\node[draw, squarednode] (1) at (-6,0.5){5};
\node[draw, squarednode] (1) at (-6,0.5){5};
\node[draw, squarednode] (2) at (-5.5,2.5){6};
\node[draw, circle,red] (3) at (-5,0.5){1};
\node[draw, squarednode] (4) at (-4.5,1.5){7};
\node[draw, circle,red] (6) at (-3.5,2.5){2};
\node[draw, circle, red] (7) at (-3,1.5){3};
\node[draw, circle, red] (9) at (-2,2.5){4};
\draw[->] (2) to (6);
\draw[->] (6) to (4);
\draw[->] (1) to (3);
\draw[->] (3) to (9);
\draw[->] (7) to (6);
\draw[->] (3) to (7);
\draw[->,blue!75] (4) to (1);
\draw[->,blue!75] (4) to (2);
\end{tikzpicture}
}
\end{figure*}

From this, we can calculate the matrix 

\[
\left[
\begin{array}{cccc|ccc}
 0 & 0   & 0   &  -1  &  0  &  0  &  0\\
 0  & 0  & 1  &  -1  &  0  & 0  & 0\\
0  &  -1  &  -1  &  0  &  -1  &  0  &  -1\\
1  &  1  &  1  &  -1  &  0  &  0  &  1\\
0  &  0  &  0  &  -1  & -1  &  0  &  0\\
0  &  0  &  1  &  -1  & 0  & -1  & 0\\
0  &  0  &  0  &  -1  &  -1  &  -1  & -1
\end{array}
\right],
\]

whose entries have mixed signs. Thus the $(\CC^\times)^3$ action on $\X$ is described as

\begin{align*}
x_1 & \to x_1,  &  x_2 &\to x_2,  & x_3&\to t_1^{-1}t_3^{-1}x_3, 
 & x_{4}&\to t_3x_4,\\
 x_{5}&\to t_1^{-1}x_{5},
& x_{6}&\to t_2^{-1}x_6,
&  x_7 &\to t_1^{-1}t_2^{-1}t_3^{-1}x_7.
\end{align*}
\end{exam}


\end{document}